\documentclass[a4paper, pdftex]{article}

\usepackage[utf8]{inputenc}
\usepackage[T1]{fontenc}
\usepackage[pdftex]{graphicx}
\usepackage{amsmath}
\usepackage{amsfonts}
\usepackage{amsthm}
\usepackage{amssymb}
\usepackage{setspace}
\usepackage{caption}
\usepackage{thmtools}
\usepackage{thm-restate}
\usepackage{mathtools}
\usepackage{hyperref}
\usepackage[capitalize,nameinlink,noabbrev]{cleveref}
\usepackage{verbatim}
\usepackage[all, cmtip]{xy}
\usepackage{multirow}
\usepackage{bm}
\usepackage[textsize = tiny]{todonotes}
\usepackage{makecell}
\usepackage{doi}

\usepackage[backend=bibtex, style=numeric, url=true, hyperref, maxbibnames=99, sorting=nyt, eprint=false, isbn=false, sortcites=true]{biblatex}
\AtEveryBibitem{\clearlist{language}}
\newtheorem*{thm*}{Theorem}
\newtheorem{theorem}{Theorem}[section]
\newtheorem{corollary}[theorem]{Corollary}
\newtheorem{lemma}[theorem]{Lemma}

\newtheorem*{fact*}{Fact}
\newtheorem{proposition}[theorem]{Proposition}

\newcounter{theoremalph}

\theoremstyle{definition}

\newtheorem{definition}[theorem]{Definition}

\theoremstyle{remark}
\newtheorem{remark}[theorem]{Remark}

	\newcommand{\ls}{\left\{}
	\newcommand{\rs}{\right\}}
	\newcommand{\sm}{\setminus}

	\newcommand{\mcP}{\ensuremath{\mathcal{P}}}

	\newcommand{\mbQ}{\ensuremath{\mathbb{Q}}}
	\newcommand{\mbR}{\ensuremath{\mathbb{R}}}
	\newcommand{\mbZ}{\ensuremath{\mathbb{Z}}}

	\newcommand{\overbar}[1]{\mkern 2mu\overline{\mkern-2mu#1\mkern-2mu}\mkern 2mu}

	\newcommand{\Aut}{\ensuremath{\operatorname{Aut}}}
	
	\newcommand{\Out}{\ensuremath{\operatorname{Out}}}

	\newcommand{\lk}{\operatorname{lk}}
	
	\newcommand{\rk}{\operatorname{rk}}

	\newcommand{\GL}[2]{\ensuremath{\operatorname{GL}_{#1}(#2)}}
	\newcommand{\SL}[2]{\ensuremath{\operatorname{SL}_{#1}(#2)}}

\newcommand{\Conebonds}{C}
\newcommand{\Ctwobonds}{C'}
\newcommand{\fibrenonone}{C^{\operatorname{non-}\operatorname{sep}}_{\subset\sigma}}
\newcommand{\fibreonlyone}{C^{\operatorname{only-}\operatorname{sep}}_{\supset\rho}}
\newcommand{\fibreallone}{C^{\operatorname{all}}_{\supset\rho}}
\newcommand{\fibrenontwo}{C'^{\operatorname{non-}2}_{\subset\tau}}
\newcommand{\fibreonlytwo}{C'^{\operatorname{only-}2}_{\supset\rho}}
\newcommand{\fibrealltwo}{C'^{\operatorname{all}}_{\supset\rho}}
\newcommand{\twosepposet}{\mathcal{P}}

\newcommand{\CV}{\operatorname{CV}}

\newcommand{\onesepedges}{E^{\operatorname{sep}}}

\newcommand{\twosepedges}{E^{2-\operatorname{sep}}}

\newcommand{\FreeS}{\mathcal FS}

\newcommand{\coeffs}{\Bbbk}

\begin{document}

\title{Bond thickenings of the simplicial boundary of Outer space}

\author{Benjamin Br\"uck}

\date{}

\maketitle

\begingroup
\renewcommand{\thefootnote}{}
\endgroup

\begin{abstract}
We study the simplicial boundary $\partial\FreeS$ of Culler--Vogtmann Outer space via thickenings defined by graph-theoretic connectivity. Let $\Ctwobonds$ be the subcomplex of the free splitting complex obtained from $\partial\FreeS$ by adding all stable graphs that are not $3$-edge connected, together with their faces. We prove that the inclusion $\partial\FreeS\hookrightarrow\Ctwobonds$ is $(2n-3)$-connected. The proof shows, more precisely, that adding graphs with cut vertices is a homotopy equivalence, while the only non-contractible fibres in the $2$-bond thickening occur over $\theta$-graphs. The result gives further evidence that $\partial\FreeS$ may be $(2n-3)$-spherical, an $\Out(F_n)$-analogue of Rognes's connectivity conjecture for the common basis complex.
It also gives a topological, universal-cover perspective that unifies several existing results about the commutative graph complex.
\end{abstract}

\section{Introduction}

\subsection{Motivation and main result}

Outer space $\CV = \CV_n$, introduced by Culler--Vogtmann
\cite{CV:Moduligraphsautomorphisms}, is a moduli space of marked metric graphs on which $\Out(F_n)$, the outer automorphism group of the free group, acts properly discontinuously.
It can be written as a union of open simplices, each of which corresponds to a marked connected graph of rank $n$ and degree at least three. However, Outer space is not a simplicial complex because some faces are missing. Adding these faces gives the simplicial closure, the free splitting complex $\FreeS$.
Its simplicial boundary $\partial \FreeS$ is the subcomplex consisting of the faces that are not contained in Outer space. Each such face corresponds to a marked graph of groups with at least one non-trivial vertex group; a more precise definition is given in \cref{sec:sphere_systems_basics}.
While both $\CV$ and $\FreeS$ are contractible \cite{CV:Moduligraphsautomorphisms, Hat:Homologicalstabilityautomorphism}, the homotopy type of $\partial \FreeS = \FreeS\setminus \CV$ is not known.

The study of the topology of $\partial \FreeS$ is motivated by an analogy with $\SL{n}{\mbZ}$ acting on the symmetric space $X = \SL{n}{\mbR}/\operatorname{SO}_n$.
Borel--Serre \cite{BS:Cornersarithmeticgroups} constructed a bordification $\bar{X}$, a contractible manifold with corners whose boundary $\partial \bar{X}$ is homotopy equivalent to the Tits building associated to $\SL{n}{\mbQ}$. This building is $(n-2)$-spherical, i.e.~homotopy equivalent to a wedge of spheres of dimension $n-2$. 
Via Poincaré--Lefschetz duality, this makes $\SL{n}{\mbZ}$ a virtual duality group, which underlies many computations of its cohomology~\cite{Lee1978, ElbazVincent2013, Sikiric2019, Brueck2024c, Brueck2024a, CFP:stabilityconjectureunstable, Brown2024a, Ash2024b, Brown2023}.
For $\Out(F_n)$, the analogous picture is less clear because $\CV$ and $\FreeS$ are not manifolds \cite{Wade2024a}.
Nonetheless, understanding the topology of $\partial \FreeS$ seems to be an important step in relating cohomological results for $\SL{n}{\mbZ}$ and $\Out(F_n)$.

One might expect $\partial \FreeS$ to be spherical of dimension roughly $n$, in analogy with the spherical building. This is consistent with the known connectivity bounds: Br\"uck--Gupta \cite[Theorem B]{BG:Homotopytypecomplex} showed that $\partial \FreeS$ is $(n-2)$-connected, i.e.~its homotopy groups vanish up to degree $n-2$,\footnote{These results were generalised to automorphism groups of freely decomposable groups \cite{Bru:buildingsfreefactora}. In this setting, Br\"uck--Piterman \cite{Brueck2024b} provided many further highly connected subcomplexes.} and Vogtmann \cite{Vogtmann2024} showed that the boundary of \emph{reduced} Outer space (see \cref{rem:connection_to_reduced}) is $(n-3)$-connected; both asked whether these bounds are sharp \cite[after Question 4.50]{Bru:buildingsfreefactora}, \cite[Section 3.2]{Vogtmann2024}. 
There is, however, growing evidence that the true dimension is larger: Br\"uck--Gupta showed that $\partial \FreeS$ is homotopy equivalent to the complex of free factor systems defined by Handel--Mosher \cite{HM:Relativefreesplitting}, which has dimension $2n-3$. This complex is a close relative of the common basis complex, a $\GL{n}{\mbZ}$-complex arising in Rognes's work on algebraic $K$-theory \cite{Rognes1992} that he conjectured to be $(2n-3)$-spherical. While it would have strong consequences, the conjecture has so far resisted proof and is only known in versions for fields \cite{Brueck2024, Hanlon, Miller2023, Piterman2025}.
One should therefore expect that it is also challenging to determine the full homotopy type of the $\Out(F_n)$-version $\partial \FreeS$, but it seems plausible that it is $(2n-3)$-spherical as well.

This expectation is supported by two recent results:
Firstly, Petersen--Wade defined a complex of non-simple disc systems on which the mapping class group of a genus-$n$ handlebody acts and showed that it is $(2n-3)$-spherical \cite[Theorem B]{Petersen2024}; for the relation to $\partial \FreeS$, see \cite[Remark 1.5]{Petersen2024}.
Secondly, Br\"uck--Miller--Piterman \cite{Brueck2025} proved that an $\Aut(F_n)$-version of $\partial \FreeS$ is $(2n-3)$-spherical.

The present work gives a further indication that $\partial\FreeS$ might indeed be $(2n-3)$-spherical:
A graph is \emph{$k$-edge connected} if it is connected and remains so after removing up to $k-1$ edges.
Let $\Ctwobonds$ be the smallest subcomplex of $\FreeS$ containing $\partial\FreeS$ and all graphs that are not $3$-edge connected. We view $\Ctwobonds$ as a thickening of the boundary; see \cref{sec:graph_complexes} for a more explicit description.
Recall that a map is called $k$-connected\footnote{Two different notions of connectivity occur in this paper: A graph being $k$-connected in the graph theoretic sense and a space or map being $k$-connected in the homotopical sense. In order to make the distinction clearer, we will say that a graph is $k$-\emph{vertex} connected if it is what is usually just called ``$k$-connected'' in graph theory.} if it induces an isomorphism on $\pi_i$ for all $i<k$ and a surjection on $\pi_k$.
The main result of this paper is as follows.

\begin{theorem}
	\label{thm:intro}
	The inclusion $\partial \FreeS \hookrightarrow \Ctwobonds$ is $(2n-3)$-connected.
\end{theorem}

In fact we prove a more precise statement: we give a filtration of this inclusion where every step is either an $\Out(F_n)$-equivariant homotopy equivalence (cf.~\cref{sec:equivariance}) or has explicitly described fibres. The only non-contractible fibres occur over marked $\theta$-graphs (graphs with two vertices and $n+1$ non-loop edges), where the relevant join in Quillen's fibre theorem is a wedge of $(n-1)!$ spheres of dimension $2n-3$; see \cref{prop:inclusion_fibres_highly_connected}.

\cref{thm:intro} does not solve the question of what the homotopy type of $\partial \FreeS$ is, but it reduces the problem to another inclusion. We have
\begin{equation*}
    \partial \FreeS \hookrightarrow \Ctwobonds \hookrightarrow \FreeS,
\end{equation*}
where $\FreeS$ is contractible and $\partial \FreeS$ is homotopy equivalent to a complex of dimension $2n-3$. \cref{thm:intro} handles the first inclusion; if the second were $(2n-3)$-connected as well, then $\partial \FreeS$ would be $(2n-3)$-spherical.

\cref{thm:intro} also has an independent payoff in terms of graph complexes: it can be seen as a topological, ``universal-cover'' version of three successive simplifications of Kontsevich's commutative graph complex -- restricting first to 2-vertex connected graphs, then to simple graphs, then to 3-vertex connected graphs. We explain this in \cref{sec:graph_complexes_intro}.

\subsection{Connection to graph complexes}
\label{sec:graph_complexes_intro}

The complex $\Ctwobonds$ contains all graphs that are not 3-edge connected, but also every graph that is obtained by collapsing a subgraph of such a graph.
One can show (see \cref{lem:Conebond_via_cut_vertices} and \cref{rem:2_sepators_in_C}) that for $n\geq 3$, a graph of $\FreeS$ is contained in $\Ctwobonds$ if and only if it is contained in $\partial \FreeS$ or it is not 3-vertex connected, which rouhgly speaking means that it can be disconnected by removing one or two vertices.\footnote{There are different conventions for vertex-connectivity in the setup of graphs with multiple edges and loops, see \cref{sec:graphs} for a more precise definition.}
In other words, $\FreeS\setminus \Ctwobonds$ consists exactly of those graphs that are 3-vertex connected.

This is closely related to results about Kontsevich's commutative graph complex:
Conant--Gerlits--Vogtmann \cite{Conant2005} showed that the homology of the commutative graph complex is the same as that of its quotient spanned by 2-vertex connected graphs, see also \cite[Appendix F]{Willwacher2015a}.
Willwacher--{\v{Z}}ivkovi{\'c} \cite{Willwacher2015} showed that the homology also does not change if one leaves out graphs with multiple edges.
Building on this, Willwacher \cite{Willwacher2025a} showed that one can furthermore restrict to 3-vertex connected graphs.
All of this reduces the size of the graph complex, which allows computations in higher ranks.

The relation between these results and \cref{thm:intro} is as follows:
Conant--Vogtmann \cite[Proposition 27 and paragraph after its proof]{CV:theoremKontsevich} showed that the homology of the commutative graph complex is the same as $H_\bullet^{\Out(F_n)}(\FreeS, \partial \FreeS; \coeffs)$, where $\coeffs$ depends on whether one considers the even or the odd version of the graph complex.
Roughly speaking, \cref{thm:intro} says that in low degrees, we can replace $(\FreeS, \partial \FreeS)$ by $(\FreeS, \Ctwobonds)$.
The chains of this pair are generated by graphs in $\FreeS\setminus \Ctwobonds$ and as observed above, these are exactly the ones that are 3-vertex connected.
In this sense, \cref{thm:intro} can be seen as a ``universal cover'' version that combines the graph complex results above. 
It is however not true that it directly implies these results; for further details, see \cref{sec:graph_complexes}.
More comments on the connection between sphericity of $\partial \FreeS$ and the commutative graph complex can also be found in \cite[Section 1.5.4]{Brueck2025}.

\begin{remark}
	There are similar results for the Lie graph complex, whose homology is isomorphic to the equivariant homology of the spine of Outer space, i.e.~the cohomology of $\Out(F_n)$ \cite[Theorem 2]{CV:theoremKontsevich}.
	Here, it is known that one can restrict to the subcomplex of 2-edge connected graphs \cite{CV:Moduligraphsautomorphisms, CV:Infinitesimaloperationscomplexes, Brun2023}.
	In upcoming work, Grego--V\'itek--Willwacher also show that for $n\leq 8$, it is sufficient to consider 3-edge connected graphs.
	They use this to compute the rational cohomology of $\Out(F_n)$ up to rank $n=8$ and give an estimate for the top-degree cohomology in rank $n=9$.
	
\end{remark}

\subsection{Structure of the proof}
\label{sec:structure_of_proof}
We prove \cref{thm:intro} in two steps: First, we show that the ``1-bond thickening'' given by the inclusion $\partial \FreeS \hookrightarrow \Conebonds$ is a homotopy equivalence, where $\Conebonds$ is the smallest subcomplex of $\FreeS$ containing all graphs that lie in $\partial \FreeS$ or are \emph{not} 2-edge connected. (Outside $\partial\FreeS$, this is the same as the poset of graphs with a cut vertex; see \cref{lem:Conebond_via_cut_vertices}.)
Then we show that the 2-bond thickening $\Conebonds \hookrightarrow \Ctwobonds$ is $(2n-3)$-connected.

\paragraph{The 1-bond thickening}
To show that the 1-bond thickening $\partial \FreeS \hookrightarrow \Conebonds$ is a homotopy equivalence (\cref{sec:one_bonds} of the present article), we define a filtration $\partial \FreeS = F_0 \subset F_1 \subset \cdots \subset F_k = \Conebonds$ and show that all intermediate inclusions are homotopy equivalences. This filtration is closely related to one defined by Conant--Gerlits--Vogtmann in \cite[Section 2]{Conant2005} who use it to study the commutative graph complex.
In this sense, this part is a topological version of their argument.
Note that Conant--Gerlits--Vogtmann in \cite[Section 3]{Conant2005} also give a topological version of the argument of \cite[Section 2]{Conant2005}: 
They show that adding graphs with cut vertices does not change the homotopy type of the boundary of the Bestvina--Feighn bordification of Outer space \cite{BF:topologyinfinity}.
That argument has a flavour that is quite different from our proof though because it requires one to describe and rescale metrics on (sequences of) graphs. In contrast to that, we argue more combinatorially, considering only the poset of cells of $\FreeS$.

\paragraph{The 2-bond thickening}
Our approach for the 2-bond thickening $\Conebonds \hookrightarrow \Ctwobonds$ (\cref{sec:overview_C'_C} until \cref{sec:finishing_proof} of the present article) is in principle similar, but technically much more involved. 
Again, we define a filtration $\Conebonds = F'_0 \subset F'_1 \subset \cdots \subset F'_k = \Ctwobonds$. We show that all but one of these inclusions are homotopy equivalences and the remaining inclusion is $(2n-3)$-connected.
For this inclusion, we get a precise description of the homotopy type of the fibre. Roughly speaking, it consists of a wedge of $(n-1)!$ spheres of dimension $2n-3$ for every marked $\theta$-graph, i.e.~graph with two vertices and $n+1$ non-loop edges connecting them, see \cref{fig:theta_graph}.
\begin{figure}
	\centering
	\includegraphics{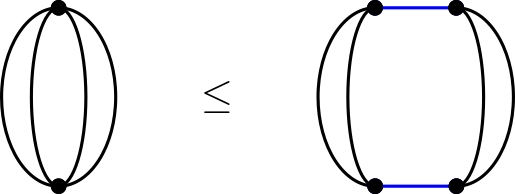}
	\caption{A $\theta$-graph in rank $n=3$ (left) and a graph with a 2-bond (in blue) that collapses to it (right). As the graph on the right is in $\Ctwobonds$, so is the $\theta$-graph.}
	\label{fig:theta_graph}
\end{figure}
For a more precise statement, see \cref{prop:inclusion_fibres_highly_connected}.
A more detailed overview of the proof idea for the inclusion $\Conebonds \hookrightarrow \Ctwobonds$ is given in \cref{sec:overview_C'_C}.

\subsection{Outlook}
The results in this paper suggest two natural generalisations:
Firstly, it seems likely that the arguments can be adapted to (moduli spaces of) graphs with marked points. The author checked that this would indeed work for the 1-bond thickening $\partial \FreeS \hookrightarrow \Conebonds$, but then decided not to include it in this article to simplify the notation and the structure of the arguments.
Secondly, one could try to use $k$-edge connectivity for $k >3$ to further thicken the boundary and get a better understanding of the inclusion $\Ctwobonds\hookrightarrow \FreeS$. 
However, there is one fundamental difficulty here: A simple graph with a vertex $v$ of degree three can never be 4-edge connected, since removing the three edges adjacent to $v$ isolates $v$ and
therefore disconnects the graph. Hence, there is no obvious generalisation to this setup and it is unclear whether one can work around this issue.

\subsection{Acknowledgments and funding information}

This work emerged from a joint project of the author with Jeremy Miller and Kevin Piterman, which also led to \cite{Brueck2025}.
I am grateful to both for numerous insightful discussions.
In particular, I thank Jeremy Miller and Peter Patzt for drawing my attention to the work of Conant--Gerlits--Vogtmann \cite{Conant2005}, which led to the proof that the inclusion $\partial\FreeS \to \Conebonds$ is a homotopy equivalence.
I also thank Kevin Piterman for his careful reading of earlier versions and for valuable explanations and references concerning (equivariant) fibre theorems.

I would like to acknowledge Maro{\v{s}} Grego, Tom{\'a}{\v{s}} V{\'i}tek and Thomas Willwacher for clarifying the results of their forthcoming work and for fruitful conversations about the connections with the present article.
Finally, I thank Richard Wade for his detailed explanations of the Robinson--Whitehouse results \cite{Robinson1996}.

The author was supported by the Deutsche Forschungsgemeinschaft through Germany’s Excellence Strategy grant EXC 2044/2–390685587 and through Project 427320536–SFB 1442.

\subsection{Declaration of generative AI use}

The author used Claude Opus 4.8 and GPT-5.5 (via gpt.uni-muenster.de, July 2026) for suggestions on organisation and wording, for background pointers, and to review proofs for possible errors or gaps; all such suggestions were independently verified by the author. In \cref{sec:graph_complexes}, Claude Opus 4.8 additionally supported exploratory discussion of the link between \cref{thm:intro} and the graph-complex results of \cite{Conant2005, Willwacher2015, Willwacher2025a}; the text was written by the author and all references and interpretations were checked against the original sources. The author takes full responsibility for all mathematical claims, proofs, citations, and conclusions.

\section{Definitions}

\subsection{Graphs}
\label{sec:graphs}
Throughout this article, all graphs are allowed to have loops and multiple edges.
Let $G$ be a graph. We write $V(G)$ for its vertex set and $E(G)$ for its edge set.
For $v\in V(G)$, we call the number of half-edges adjacent at $v$ the \emph{degree} or \emph{valence} of $v$. The degree of a graph $G$ is the minimal degree of its vertices.
If $G$ is connected, we denote by $\rk(G)$ the rank of its fundamental group, i.e.~the first Betti number of $G$.

When we want to stress that we consider a graph as a topological space, we write $|G|$ for its geometric realisation. For $e\in E(G)$, we consider $e\subset |G|$ as the open edge corresponding to $e$ without its endpoints and we write $\overline{e}\subseteq |G|$ for the closed edge corresponding to $e$ together with its endpoints. We consider every vertex $v\in V(G)$ as a point in $|G|$.

If $E\subseteq E(G)$, we denote by $G-E$ the graph that is obtained by deleting all edges from $E$ (but keeping their endpoints).
If $X$ is a set of vertices and open or closed edges of $|G|$, we write $|G| - X$ for the topological space obtained by removing all these vertices and edges from  the geometric realisation of $G$; in general, this is not a graph itself.

A subset $S\subseteq E(G)$ is a \emph{$k$-bond} if it has size $k$ and $G-S$ is disconnected. 
An edge that forms a 1-bond is called a \emph{separating edge} or \emph{bridge}.
A graph is \emph{$k$-edge connected} if it does not contain an $l$-bond for any $l < k$. We only use this notion for $k\in \ls 2,3\rs$ in the present article.

A \emph{cut vertex} is a vertex $v\in V(G)$ such that $|G|-v$ is disconnected.
A pair of vertices $\ls v, w\rs\in V(G)$ is called a \emph{2-vertex cut} if $v\neq w$ and
either $v$ and $w$ are the endpoints of a multiedge or $|G|$ becomes disconnected after removing $v,w$ and the edge connecting them (if there is one).\footnote{There are different conventions for defining vertex connectivity for multi-graphs, it is common to see the endpoints of multi-edges as 2-vertex cuts, see e.g.~\cite[Section 1.10]{Diestel2000}.}
Note that a graph that is $3$-vertex connected is nec
We say that a connected graph is \emph{$2$-vertex connected} if it does not contain a cut vertex and is \emph{$3$-vertex connected} if it is contains neither a cut vertex nor a 2-vertex cut, see \cref{fig:example_2_vertex_cut}.essarily simple because the endpoint of every loop forms a cut vertex and the endpoints of every multi-edge form a $2$-vertex cut.
\begin{figure}
	\centering
	\includegraphics{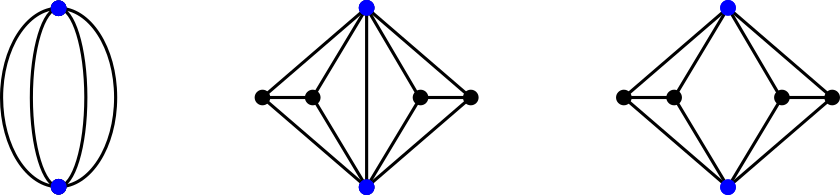}
	\caption{Three graphs with a 2-vertex cut, marked in blue.}
	\label{fig:example_2_vertex_cut}
\end{figure}

One type of graph that will be important in this article is the \emph{rank-$n$ $\theta$-graph}, which is the graph with two vertices and $n+1$ non-loop edges connecting them as depicted in \cref{fig:theta_graph}. For $n\geq 2$, it is $3$-edge connected, but not $3$-vertex connected as its two vertices form a 2-vertex cut.

\subsection{(Poset) topology}
We often identify simplicial complexes with their posets of simplices.
If $X$ is a poset and $x\in X$, we write $X_{<x}$ for the subposet of all $y\in X$ with $y < x$. Similarly, we define $X_{\leq x}$, $X_{>x}$ and $X_{\geq x}$.

If $X$ and $Y$ are posets, we call a map $f:X\to Y$ a \emph{poset map} if for all $x\leq x'\in X$, we have $f(x) \leq f(x')$ in $Y$.
A poset map $f:X\to X$ is called \emph{monotone} if either for all $x\in X$, we have $f(x)\leq x$ or for all $x\in X$, we have $f(x)\geq x$. Monotone poset maps define homotopy equivalences, which we will use frequently.
Another tool that we will use is Quillen's fibre theorem, which we will use in the following form, stated in \cite[Theorem 5.3]{Piterman2024}.
\begin{theorem}
	\label{thm:Quillen_fibre}
	Let $f\colon X \to Y$ be a poset map between posets of finite height. Let $k\geq -1$ and assume that for all $y\in Y$, the join $f^{-1}(Y_{\leq y})\ast Y_{> y}$ is $k$-connected. Then $f$ is $(k+1)$-connected, i.e.~it induces an isomorphism on $\pi_i$ for all $i\leq k$ and a surjection on $\pi_{k+1}$.
\end{theorem}

Note that order complexes of posets are CW-complexes, so if a poset map is $k$-connected for all $k\geq 0$, then it is a homotopy equivalence.

\subsection{Complexes of sphere systems}
\label{sec:sphere_systems_basics}
Starting from this section and for the rest of the paper, we fix $n\geq 2$.
Let $M = M_n$ be the $n$-fold connected sum of $S^1\times S^2$.
A \emph{sphere system} in $M$ is a finite set of isotopy classes of disjointly embedded 2-spheres that do not bound a 3-disk.
These sphere systems first appeared in work of Whitehead \cite{Whi:CertainSetsElements,Whitehead1936} and were generalised by Hatcher \cite{Hat:Homologicalstabilityautomorphism}.
We write $\FreeS = \FreeS(M)$ for the simplicial complex whose $p$-simplices are given by sphere systems in $M$ of size $p+1$ and where the face relation is given by containment.
If $\rho = \{ [S_1],\ldots, [S_r]\}\in \FreeS(M)$ is a sphere system, we can choose representatives $S_i$ and open tubular neighbourhoods $S_i\subset T_i$ in the interior of $M$ such that $T_i\cap T_j = \emptyset$ for $i\neq j$. The complement $M \setminus \bigcup_i T_i$ is a (possibly disconnected) 3-manifold with boundary and its diffeomorphism type depends only on $\rho$, not on the choices of representatives $S_i$ and open tubular neighbourhoods $T_i$. For an example, see \cite[Figure 1]{Brueck2025}.
We denote it by $M- \rho$. In what follows, we will usually identify isotopy classes of spheres $[S]$ with their representatives $S$.

Let $\partial \FreeS\subset \FreeS$ be the subposet consisting of sphere systems $\sigma$ where at least one connected component of $M_{n}-\sigma$ is not simply connected. It is easy to see that this subposet is downwards-closed, so in particular forms a sub\emph{complex} of $\FreeS$.
It is called the \emph{simplicial boundary of Outer space $CV_n$}.

\begin{remark}
	\label{rem:connection_to_reduced}
	There is a reduced version of Outer space, denoted by $\CV^r$, which is obtained from $\CV$ by removing all graphs that are not 2-edge connected, i.e.~contain a separating edge. Let $\FreeS^r$ and $\partial \FreeS^r$ denote its simplicial closure and simplicial boundary, respectively.
	Equivalently, $\FreeS^r$ is the subcomplex of $\FreeS$ given by all sphere systems $\sigma$ such that $M-\ls S\rs$ is connected for all $S\in \sigma$.
	It is true that $\CV$ deformation retracts to $\CV^r$ and that there is an inclusion of pairs $(\FreeS^r, \partial \FreeS^r)\hookrightarrow (\FreeS, \partial \FreeS)$.
	However, this does not induce a homotopy equivalence between $\partial \FreeS^r$ and $\partial \FreeS$ and already for $n=2$, these spaces are not homotopy equivalent.
	Vogtmann \cite{Vogtmann2024} proved that $\partial \FreeS^r$ is homotopy equivalent to the boundary of jewel space, defined by Bux--Smillie--Vogtmann \cite{BSV:bordificationouterspace}, which in turn is homotopy equivalent to the boundary of the Bestvina--Feighn bordification of Outer space \cite{BF:topologyinfinity}.
\end{remark}

\subsubsection{Dual graphs}

To every $\sigma\in \FreeS$, we can associate a \emph{dual graph} $\Gamma(\sigma)$ (cf.~\cite[Appendix]{Hat:Homologicalstabilityautomorphism}).
Its vertices are the connected components of $M - \sigma$.
There is one edge $e_S$ for every sphere $S\in \sigma$ that connects the two vertices corresponding to the connected components on the two sides of $S$. Note that it is possible that both sides of $S$ lie in the same connected component; in this case, $e_S$ is a loop. 
If $G$ is a graph and $X\subseteq \FreeS$ is a subposet, then we sometimes write $G\in X$ if there is $\sigma\in X$ with $G = \Gamma(\sigma)$. 
With this notation, we say that a graph $G$ is \emph{stable} if $G\in \FreeS$.\footnote{Often, one equips dual graphs of systems $\sigma\in \partial \FreeS$ with an additional vertex decoration that associates to each vertex the rank of the fundamental group of the corresponding component of $M-\sigma$, see e.g.~\cite[Section 6]{Brueck2025}. These do not show up in this article because we will only need to consider graphs that are dual to sphere systems in $\FreeS\setminus \partial \FreeS$.} 
We have $G\in \FreeS\setminus \partial \FreeS$ if and only if $G$ is connected, has fundamental group isomorphic to $F_n$ and $G$ has degree at least three.

If $\rho\subset \sigma$ is a face of $\sigma\in \FreeS$, then $\Gamma(\rho)$ is obtained from $\Gamma(\sigma)$ by collapsing the edges corresponding to the spheres in $\sigma\setminus \rho$, see \cref{fig:dual_graph_example}.
We then write $\Gamma(\rho)\leq \Gamma(\sigma)$.
\begin{figure}
	\centering
	\includegraphics{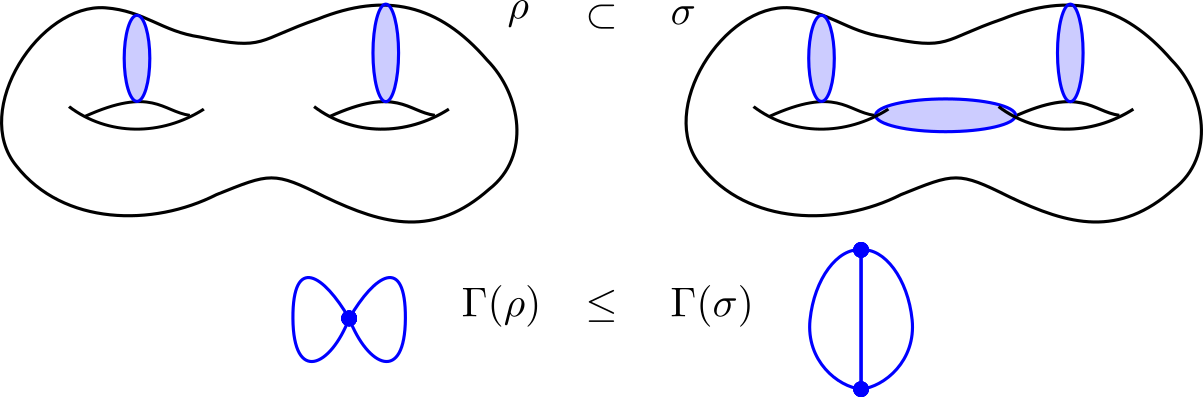}
	\caption{Two sphere systems $\rho\subset \sigma$ in $\FreeS\setminus \partial \FreeS$ and their dual graphs.}
	\label{fig:dual_graph_example}
\end{figure}

\subsubsection{The link of a sphere system}
\label{sec:link_sphere_system}

The link $\lk_{\FreeS}(\rho)$ of a sphere system $\rho\in\FreeS$ is a simplicial complex whose vertices are isotopy classes of spheres $S$ such that $S\not \in \rho$ and $\rho \cup \ls S \rs \in \FreeS$. 
As $\FreeS$ is a flag complex, so is $\lk_{\FreeS}(\rho)$. This means that a collection of such spheres forms a simplex (in $\FreeS$ or in $\lk_{\FreeS}(\rho)$) if and only if for all $S,T$ in the collection, we have $\rho\cup \ls S,T \rs\in \FreeS$. There are isomorphisms between $\FreeS_{\supset\rho}$ and $\lk_{\FreeS}(\rho)$ given by removing, respectively adding $\rho$.

The link decomposes as an iterated join: We have
\begin{equation*}
	\lk_{\FreeS}(\rho) = \ast_{v\in V(G)} \FreeS_v,
\end{equation*}
where $G\coloneqq \Gamma(\rho)$ and for $v\in V(G)$, $\FreeS_v$ is the complex of all sphere systems $\sigma_v$ in the connected component $M_v\subseteq M-\rho$ corresponding to $v$ such that $\rho \cup \sigma_v\in \FreeS$.

If $\rho\not\in \partial \FreeS$, then for all $v\in V(G)$, the complex $\FreeS_v$ is isomorphic to a complex of partitions: 
Write $E(v)$ for the set of half-edges adjacent to $v$.
As $\rho\not\in \partial \FreeS$, the connected component $M_v\subseteq M-\rho$ is a 3-sphere with one boundary component for each half-edge in $E(v)$. 
Every sphere $S\in \FreeS_v$ induces a partition $P_S$ of $E(v)$ into two subsets, each of which contains at least two elements (the latter is forced by the condition that the degree of all graphs in $\FreeS$ is at least 3).
Two spheres $S,T\in \FreeS_v$ form a simplex in $\FreeS_v$ if and only if the partitions are compatible, which means that one side of $P_S$ is contained in one side of $P_{T}$. $\FreeS_v$ is the flag complex defined by this edge relation.

\subsection{Bond thickenings of the boundary \texorpdfstring{$\partial \FreeS$}{\partial S}}

We define two additional posets, which are nested between $\partial \FreeS$ and $\FreeS$,
\begin{equation*}
	\partial \FreeS \hookrightarrow \Conebonds \hookrightarrow \Ctwobonds \hookrightarrow \FreeS.
\end{equation*}

\begin{definition} 
	\begin{enumerate}
		\item Let $\Conebonds\subset \FreeS$ be the smallest downwards-closed subposet of $\FreeS$ that contains $\partial \FreeS$ and every $\sigma \in \FreeS$ such that  $\Gamma(\sigma)$ has a 1-bond, i.e.~contains a separating edge.
		\item Let $\Ctwobonds\subset \FreeS$ be the smallest downwards-closed subposet of $\FreeS$ that contains $\Conebonds$ and every $\sigma \in \FreeS$ such that $\Gamma(\sigma)$ has a 2-bond.
	\end{enumerate}
\end{definition}

In other words, $\Conebonds$ is the smallest subcomplex of $\FreeS$ containing all graphs that lie in $\partial \FreeS$ or are \emph{not} 2-edge connected;
$\Ctwobonds$ is the smallest subcomplex of $\FreeS$ containing all graphs that lie in $\partial \FreeS$ or are \emph{not} 3-edge connected.

By definition, these posets are downwards-closed, so in fact, they form subcomplexes of $\FreeS$.
For reference, we record the following immediate consequence.

\begin{lemma}
	\label{lem:rk_preserved_2_connected}
	Let $\rho,\sigma \in \FreeS$ with $\rho\subseteq \sigma$.
	\begin{enumerate}
		\item If $\rho \in \FreeS\setminus \partial \FreeS$, so is $\sigma$.
		\item If $\rho \in \FreeS\setminus \Conebonds$, so is $\sigma$.
	\end{enumerate}
\end{lemma}

\section{The 1-bond thickening $\partial \FreeS\hookrightarrow \Conebonds$}
\label{sec:one_bonds}
In this section, we show the following:

\begin{theorem}
	\label{thm:hom_eq_partialS_C}
	The inclusion $\partial \FreeS\hookrightarrow \Conebonds$ is a homotopy equivalence.
\end{theorem}

The proof of this can be seen as a warmup for the more involved case of the inclusion $\Conebonds \hookrightarrow \Ctwobonds$ that we consider in later sections.

\subsection{Graphs with bridges and cut vertices}
\label{sec:graphs_with_1_bonds}
We start with results about (graphs with) bridges and cut vertices.
We are brief here because very similar arguments appear again in the more complicated setting of 2-bonds in \cref{sec:graphs_with_2_bonds}, where we spell them out in detail.

\begin{lemma}
	\label{lem:Conebond_via_cut_vertices}
	Let $\sigma \in \FreeS$. Then $\sigma\in \Conebonds$ if and only if $\sigma \in \partial \FreeS$ or $\Gamma(\sigma)$ is not 2-vertex connected, i.e.~it has a cut vertex.
\end{lemma}
\begin{proof}
	Write $G = \Gamma(\sigma)$.
	For the first direction, let $\sigma\in \Conebonds\setminus \partial \FreeS$. We need to show that $\sigma$ contains a cut vertex. By definition, there is $\sigma \subseteq \tau\in \Conebonds$ such that $\tau$ is either in $\partial \FreeS$ or $G'\coloneqq \Gamma(\tau)$ has a separating edge. The first case cannot occur by \cref{lem:rk_preserved_2_connected}.
	So assume that $G'$ has a separating edge $e$. As $G'$ has degree at least three, every component of $G'-e$ has non-trivial fundamental group.
	As $G$ and $G'$ have isomorphic fundamental groups, this implies that none of these components can be collapsed to a single vertex in $G$. It follows that $G$ has a cut vertex: if $e$ is not collapsed, then any of its endpoints is a cut vertex and otherwise, the image of $e$ under the collapse map is a cut vertex.

	For the converse, first note that if $\sigma\in \partial \FreeS$, then clearly also $\sigma \in \Conebonds$. So assume that $\sigma \in\FreeS\setminus \partial \FreeS$ is such that $G\coloneqq \Gamma(\sigma)$ has a cut vertex $v$. We need to show that there is $G\leq G'\in\Conebonds$ that contains a separating edge.
	If a connected component of $|G|\setminus \ls v \rs$ is adjacent to $v$ with only one half-edge, then this half-edge is the initial part of a separating edge and we are done.
	But if every component of $|G|\setminus \ls v \rs$ is adjacent to $v$ with at least two half-edges, we can introduce a separating edge at $v$ that separates one of the components from the other ones. The result is a graph $G'$ with a separating edge that has fundamental group $F_n$ and -- because of the condition of the half-edges -- degree at least three, so $G\leq G'\in \Conebonds$.
\end{proof}

For a graph $G$, we let $\onesepedges(G)\subseteq E(G)$ be the \emph{set of separating edges} in $G$. 
For $\sigma \in \FreeS\setminus \partial \FreeS$, we let $\onesepedges(\sigma)$ be the subset of $\sigma$ corresponding to the edges that form separating edges in the dual graph,
\begin{equation*}
	\onesepedges(\sigma) = \ls S\in \sigma \mid e_S \in \onesepedges(\Gamma(\sigma)) \rs.
\end{equation*}
We record two elementary observations about these sets of separating edges.

\begin{lemma}
	\label{lem:collapsing_not_to_partial_preserves_bridges}
	Let $\rho, \sigma \in \FreeS\setminus \partial \FreeS$ with $\rho \subseteq \sigma$. Then $\onesepedges(\rho) = \onesepedges(\sigma) \cap \rho$.
\end{lemma}
\begin{proof}
	This is easy to see using that $\Gamma(\rho)$ is obtained from $\Gamma(\sigma)$ by collapsing the edges corresponding to $\sigma\setminus \rho$.
\end{proof}

\begin{lemma}
	\label{lem:bridges_are_a_forest}
	If $\sigma \in \FreeS\setminus \partial \FreeS$, then $\onesepedges(\sigma)$ forms a forest inside $\Gamma(\sigma)$.
\end{lemma}
\begin{proof}
	This follows from the following observation: If $G$ is a graph and $e\in \onesepedges(G)$ is a separating edge, then $G-e$ has one edge less than $G$ and one connected component more. Hence, the first Betti numbers of the graphs agree, $b_1(G-e) = b_1(G)$. Furthermore, we have $\onesepedges(G-e)  = \onesepedges(G)\setminus\ls e \rs$, so the claim follows by induction.
\end{proof}

\subsection{Filtration by separating edges}

In order to prove \cref{thm:hom_eq_partialS_C}, we define a filtration of $\Conebonds$ by subposets $C_{p,q}$, where $p$ and $q$ are non-negative integers.
This can be seen as a homotopical version of the spectral sequence argument from \cite[Section 2]{Conant2005}.

\begin{definition}
	\label{def:Conebonds}
	Let $\Conebonds_{p,q}$ be the subposet of $\Conebonds$ consisting of all $\sigma$ such that 
		\begin{enumerate}
			\item $\sigma\in \partial \FreeS$, or
			\item $|\onesepedges(\sigma)|> p$, or
			\item $|\onesepedges(\sigma)| = p$ and $|\sigma| \leq p+q$.
		\end{enumerate}
\end{definition}

In words, $C_{p,q}$ consists of those $\sigma\in \Conebonds$ that either lie in $\partial \FreeS$, contain more than $p$ separating spheres, or contain exactly $p$ separating but at most $q$ non-separating spheres.

\begin{lemma}
	\label{lem:filtration_of_Conebonds}
	\begin{enumerate}		
		\item \label{it:Cpq_partial}If $p> 2n-3$, then $\Conebonds_{p,q} = \partial \FreeS$.
		\item \label{it:Cpq_Cplus1_q}If $q\leq n-1$, we have $\Conebonds_{p,q} = \Conebonds_{p+1,3n-4-p}$.
		\item \label{it:C0q}If $q\geq 3n-4$, then $\Conebonds_{0,q} = \Conebonds$.
	\end{enumerate}
\end{lemma}
\begin{proof}
	For every $\sigma\in \FreeS\setminus \partial \FreeS$, the dual graph $\Gamma(\sigma)$ has at most $3n-3$ edges or (equivalently) at most $2n-2$ vertices (see e.g.~\cite{CV:Moduligraphsautomorphisms}).
	This implies that $\Conebonds_{0,q} = \Conebonds$ for $q\geq 3n-3$. The slightly sharper bound in \cref{it:C0q} follows because if $\sigma \in \Conebonds_{0,q}\setminus \partial \FreeS$, then $\Gamma(\sigma)$ has at least one cut vertex by \cref{lem:Conebond_via_cut_vertices}, but no separating edge. The latter implies that every cut vertex needs to have valence at least 4, so $\Gamma(\sigma)$ has at most $3n-4$ edges.

	By \cref{lem:bridges_are_a_forest}, for every $\sigma\in \Conebonds\setminus \partial \FreeS$, the subset $\onesepedges(\sigma)$ forms a forest in $\Gamma(\sigma)$. This implies that 
	\begin{equation*}
		|\onesepedges(\sigma)| \leq |V(\Gamma(\sigma))| - 1 \leq 2n-3 \text{ and } |\sigma|-|\onesepedges(\sigma)|\geq n.
	\end{equation*}
	From these, \cref{it:Cpq_partial} and \cref{it:Cpq_Cplus1_q} follow.
\end{proof}

The previous lemma implies that we get a filtration
\begin{multline*}
	\partial \FreeS = \Conebonds_{2n-2,\bullet} = \Conebonds_{2n-3,n-1} \subseteq \Conebonds_{2n-3,n} =  \Conebonds_{2n-4,n-1} \subseteq \cdots
	\\
	 \cdots \subseteq \Conebonds_{1,3n-4} = \Conebonds_{0,n-1} \subseteq \cdots \subseteq \Conebonds_{0,3n-5} \subseteq \Conebonds_{0,3n-4} = \Conebonds.
\end{multline*}
Note that every inclusion here is of the form $f:C_{p,q-1}\hookrightarrow C_{p,q}$.
In the first part, $\partial \FreeS \hookrightarrow \Conebonds_{1,3n-4} = \Conebonds_{0,n-1}$, all graphs in $\FreeS$ that are not 2-edge connected are added, starting with the ones that have the biggest number of separating edges. In the second part, $\Conebonds_{0,n-1} \hookrightarrow \Conebonds$, the graphs that are 2-edge connected but not 2-vertex connected are added.
To prove \cref{thm:hom_eq_partialS_C}, we will show that all of the inclusions above are homotopy equivalences.

\subsection{The fibres}

To prove that $f:C_{p,q-1}\hookrightarrow C_{p,q}$ is a homotopy equivalence, we want to use Quillen's fibre theorem as stated in \cref{thm:Quillen_fibre}. We hence need to study the connectivity of the fibres, which is what we do in the following two lemmas.
In their proofs, we use that $\sigma\in \Conebonds_{p,q}\setminus \Conebonds_{p,q-1}$ if and only if the following hold:
\begin{equation*}
	\sigma\in \FreeS\setminus \partial \FreeS,\, |\twosepedges(\sigma)|=p \text{ and } |\sigma| = p+q.
\end{equation*}

\begin{lemma}
	\label{lem:lower_fibre_one_bonds}
	If $p\geq 1$, then for all $\sigma\in \Conebonds_{p,q}\setminus \Conebonds_{p,q-1}$, the lower fibre $f_{\subset\sigma}\coloneqq f^{-1}((\Conebonds_{p,q})_{\subseteq \sigma})$ is contractible.
\end{lemma}
\begin{proof}
	The poset $f_{\subset\sigma}$ is given by all $\rho\subseteq\sigma$ that lie in $C_{p,q-1}$.
	Let $\fibrenonone\subseteq \FreeS_{\subset\sigma}$ be the subposet consisting of all sphere systems that can be obtained from $\sigma$ by removing a non-empty set of non-separating spheres,
	\begin{equation*}
		\fibrenonone\coloneqq \ls \emptyset\neq\rho\subset \sigma \mid \onesepedges(\sigma)\subseteq \rho\rs.
	\end{equation*}
	We will prove the claim by showing the following three things: $\fibrenonone$ is a subposet of $f_{\subset\sigma}$; it is a deformation retraction of $f_{\subset\sigma}$, so in particular has the same homotopy type; it is contractible if $p\geq 1$.

	For the first point, let $\rho \in \fibrenonone$. We need to show that $\rho \in C_{p,q-1}$. If $\rho\in \partial \FreeS$, then this is true by definition.
	If $\rho\not\in \partial \FreeS$, then by \cref{lem:collapsing_not_to_partial_preserves_bridges}, we have $\onesepedges(\rho) = \onesepedges(\sigma)\cap \rho = \onesepedges(\sigma)$.
	Hence, $|\onesepedges(\rho)| = |\onesepedges(\sigma)| = p$ and as $\rho\subset \sigma$, we have $|\rho| < |\sigma| = p+q$, so $\rho \in C_{p,q-1}$.
	
	For the second point, let $\rho\in f_{\subset\sigma}$. Then $\rho\in C_{p,q-1}$, so it either has (at least) $p$ separating spheres or it lies in $\partial \FreeS$.
	If it has $p$ separating spheres, then by \cref{lem:collapsing_not_to_partial_preserves_bridges}, we have $\onesepedges(\rho)= \onesepedges(\sigma)$, so $\rho$ is obtained from $\sigma$ by removing a non-empty set of non-separating spheres (i.e.~$\rho\in \fibrenonone$).
	If on the other hand $\rho\in \partial \FreeS$, then by \cref{lem:bridges_are_a_forest}, there must be at least one non-separating sphere of $\sigma$ that is not contained in $\rho$.
	In either case, we get $\sigma\setminus\rho\not\subseteq \onesepedges(\sigma)$.
	This implies that the assignment $\rho \mapsto \rho\cup \onesepedges(\sigma)$ defines a monotone poset map $f_{\subset\sigma} \to \fibrenonone$. This poset map is the identity on $\fibrenonone$ and hence defines a deformation retraction.

	For the third point, note that \cref{lem:bridges_are_a_forest} implies that $\onesepedges(\sigma)\in \partial \FreeS$. As $\sigma\not\in \partial \FreeS$, we have $\onesepedges(\sigma)\subset\sigma$.
	But as $p\geq 1$, we also have $\onesepedges(\sigma)\neq \emptyset$, so this forms a minimal element in $\fibrenonone$, which hence is contractible.
\end{proof}

The above lemma implies that the inclusion $\partial \FreeS \hookrightarrow \Conebonds_{1,3n-4} = \Conebonds_{0,n-1}$ that thickens $\partial \FreeS$ by adding the graphs that are not 2-edge connected, is a homotopy equivalence.
For the second thickening $\Conebonds_{0,n-1} \hookrightarrow \Conebonds$ that adds the graphs that are 2-edge connected but not 2-vertex connected, we need the following lemma.

\begin{lemma}
	\label{lem:upper_fibre_one_bonds}
	If $p=0$, then for all $\rho\in \Conebonds_{p,q}\setminus \Conebonds_{p,q-1}$, the upper interval $\fibreallone\coloneqq (\Conebonds_{p,q})_{\supset \rho}$ is contractible.
\end{lemma}
\begin{proof}
	Let $\fibreonlyone\subseteq \FreeS_{\supset\rho}$ be the subposet consisting of all elements that can be obtained from  $\rho$ by adding a non-empty set of separating spheres,
	\begin{equation*}
		\fibreonlyone\coloneqq \ls \rho \subset \sigma\in \FreeS \mid \sigma\setminus \rho \subseteq \onesepedges(\sigma) \rs.
	\end{equation*}
	Again, we prove the claim in three steps: We show that $\fibreonlyone$ is contained in $\fibreallone$, that it is a deformation retraction and that it is contractible if $p=0$.

	First assume that $\sigma \in \fibreonlyone$. \cref{lem:rk_preserved_2_connected} implies that $\sigma\not\in \partial \FreeS$, so by \cref{lem:collapsing_not_to_partial_preserves_bridges}, we have $|\onesepedges(\sigma)|>|\onesepedges(\rho)| = p$ and we get $\sigma\in \Conebonds_{p,q}$. This shows that $\fibreonlyone\subseteq \fibreallone$.

	For the second point, assume that $\sigma\in \fibreallone$. Again by \cref{lem:rk_preserved_2_connected}, we have $\sigma\not \in \partial \FreeS$. As $\sigma\in C_{p,q}$ and $|\sigma|>|\rho| = p+q$, we have $|\onesepedges(\sigma)|>p$. But by \cref{lem:collapsing_not_to_partial_preserves_bridges}, we have $\onesepedges(\rho) = \onesepedges(\sigma)\cap \rho$, so as $|\onesepedges(\rho)|=p$, we have $\onesepedges(\sigma)\not \subseteq \rho$. 
	This allows us to define a monotone poset map $\fibreallone \to \fibreonlyone$ by removing all non-separating spheres that are not contained in $\rho$, i.e.~sending $\sigma$ to $\rho \cup \onesepedges(\sigma)$. That this image is indeed an element of $\fibreonlyone$ and that the map is monotone follows from \cref{lem:collapsing_not_to_partial_preserves_bridges}.
	This poset map is the identity on $\fibreonlyone$ and hence defines a deformation retraction.

	For the last step, we first observe that $\fibreonlyone$ is a full subcomplex of $\lk_{\FreeS}(\rho) \cong \FreeS_{\supset \rho} $ (see \cref{sec:link_sphere_system} for a description of this complex): 
	If $\tau \in \fibreonlyone$ and $\rho \subset \sigma\subseteq \tau$, then 
	\begin{equation*}
		\sigma \setminus \rho = (\tau\setminus \rho)\cap \sigma \subseteq \onesepedges(\tau)\cap \sigma = \onesepedges(\sigma),
	\end{equation*}
	where the inclusion uses that $\tau\in \fibreonlyone$ and the last equality follows from \cref{lem:collapsing_not_to_partial_preserves_bridges}. Hence, we have $\sigma\in \fibreonlyone$. This proves that $\fibreonlyone$ is downwards closed in $\FreeS_{\supset \rho}$ and hence a subcomplex. To see that it is full, assume that we have $\tau \in \FreeS_{\supset \rho}$ such that $\rho\cup \ls S \rs\in \fibreonlyone$ for all $S\in \tau\setminus \rho$. Then for all such $S$, we get
	\begin{equation*}
		S\in \onesepedges(\rho\cup \ls S\rs) = \onesepedges(\tau)\cap \left(\rho\cup \ls S \rs\right),
	\end{equation*}
	where the equality follows again from \cref{lem:collapsing_not_to_partial_preserves_bridges}. Hence, we have $\tau\setminus \rho \subseteq \onesepedges(\tau)$, i.e.~$\tau\in \fibreonlyone$. This proves that $\fibreonlyone$ is a full subcomplex of $\FreeS_{\supset \rho}$.

	As a consequence, the join decomposition of $\lk_{\FreeS}(\rho)\cong \FreeS_{\supset \rho}$ described in \cref{sec:link_sphere_system} induces a join decomposition of $\fibreonlyone$ that is of the form 
	\begin{equation*}
		\fibreonlyone \cong \ast_{v\in V(G)} \mcP_v,
	\end{equation*}
	where $G= \Gamma(\rho)$ is the dual graph and for a vertex $v\in G$ (i.e.~a component of $M-\rho$), $\mcP_v$ is the subcomplex of all sphere systems $\sigma_v$ in $v$ such that $\rho \cup \sigma_v\in \fibreonlyone$.
	To show that $\fibreonlyone$ is contractible, it hence suffices to show that there is some $v\in V(G)$ such that $\mcP_v$ is contractible.
	That this is the case if $p=0$, i.e.~$G$ has a cut vertex but no separating edge, is shown in \cite[Proof of Lemma 2.1]{Conant2005}.\footnote{The complex $\mcP_v$ is denoted by $X^v$ in that article. It is the poset of all graphs that are obtained from $G$ by blowing up $v$ to a tree of separating edges. The authors describe it in terms of a subcomplex of the complex $\FreeS_v$ of partitions that already showed up in \cref{sec:link_sphere_system}. They then show that this subcomplex is contractible by showing that it has a cone point.}
\end{proof}

\begin{proof}[Proof of \cref{thm:hom_eq_partialS_C}]
	It is enough to show that for all $p$ and $q$, the inclusion $f:C_{p,q-1}\hookrightarrow C_{p,q}$ is a homotopy equivalence.
	To show this, we apply \cref{thm:Quillen_fibre}: We claim that for every $\sigma\in C_{p,q}$, the join 
	\begin{equation}
	\label{eq:fibre_join_one_bonds}
		f^{-1}((C_{p,q})_{\subseteq \sigma})\ast (C_{p,q})_{\supset \sigma}
	\end{equation}
	is contractible.
	If $\sigma\in C_{p,q-1}$, this is clear because then $f^{-1}((C_{p,q})_{\subseteq \sigma})$ has the maximal element $\sigma$. 
	For $\sigma\in C_{p,q}\sm C_{p,q-1}$, the claim follows from \cref{lem:lower_fibre_one_bonds} and \cref{lem:upper_fibre_one_bonds}.
	Hence, the join in \cref{eq:fibre_join_one_bonds} is contractible for all $\sigma\in C_{p,q}$, so by \cref{thm:Quillen_fibre}, the inclusion $f:C_{p,q-1}\hookrightarrow C_{p,q}$ is $k$-connected for all $k\geq 0$, so a homotopy equivalence. 
\end{proof}

\section{Strategy for the 2-bond thickening $\Conebonds\hookrightarrow \Ctwobonds$}
\label{sec:overview_C'_C}
\label{sec:intuition}

We now turn to the more involved case of the inclusion $\Conebonds\hookrightarrow \Ctwobonds$. Our aim here is to show the following:

\begin{theorem}
\label{thm:C'_he_C}
The inclusion $\Conebonds\hookrightarrow \Ctwobonds$ is $(2n-3)$-connected.
\end{theorem}

In fact, we will show a slightly stronger statement, which also gives an explicit description of the homotopy type of the fibres of the inclusion $\Conebonds\hookrightarrow \Ctwobonds$, see \cref{prop:inclusion_fibres_highly_connected}.

The idea of the argument is similar to the one we used for the inclusion $\partial \FreeS \hookrightarrow \Conebonds$ in the previous section, but we are faced with additional difficulties here.
We start with an overview before we get into the details. 

Again, we use a filtration to subdivide the inclusion $\Conebonds\hookrightarrow \Ctwobonds$ into several steps.
What we would like to have is that all the inclusions arising in this filtration are homotopy equivalences. This will not be true, but we will show that \emph{all but one} of the inclusions are homotopy equivalences and that one inclusion is $(2n-3)$-connected and has explicitly described fibres.

We build the filtration such that if $X\hookrightarrow Y$ is one of the intermediate inclusions, then $Y$ is obtained by adding all sphere systems $\sigma$ such that $\Gamma(\sigma)$ is isomorphic to one of a fixed finite set of graphs that only depends on the filtration level. Then for showing that $f: X\hookrightarrow Y$ is a homotopy equivalence (or highly connected), we will show that for all $\sigma\in Y\setminus X$, the join of $f^{-1}(Y_{\subseteq \sigma}) = X_{\subset \sigma}$ and $Y_{\supset \sigma} = X_{\supset \sigma}$ is contractible (or highly connected).

To get such a filtration, we hence need to order the graphs that arise in $\Ctwobonds\setminus \Conebonds$ appropriately.
Our intuition here is that we order the graphs by how connected they are, starting with the least connected ones.
These are the graphs in $\Conebonds$, which all lie in the boundary or have a cut vertex, so are not 2-vertex connected (see \cref{lem:Conebond_via_cut_vertices}). 
The graphs in $\Ctwobonds\setminus \Conebonds$ are 2-vertex connected, but they all contain a 2-vertex cut, so they are not 3-vertex connected (see \cref{rem:2_sepators_in_C}).
Intuitively speaking, we now order these graphs by how 3-edge connected they are, starting with graphs that have many 2-bonds and then successively adding the ones with fewer and fewer 2-bonds.

For the fibres, this means that $X_{\subset \sigma}$ should be the poset of all $\rho\subset \sigma$ such that $\Gamma(\rho)$ does not have less 2-bonds than $\Gamma(\sigma)$. 
This suggests that $X_{\subset \sigma}$ should be (at least up to homotopy) the poset of all ways of collapsing subgraphs of $\Gamma(\sigma)$ that do not contain edges in 2-bonds.
Similarly, $X_{\supset \sigma}$ should be (at least up to homotopy) the poset of all ways of adding 2-bonds in $\Gamma(\sigma)$.

What makes the situation here trickier than its 1-bond analogue is that distinct 2-bonds can share edges. This makes it hard to keep track of the number of 2-bonds.
We work around this by not actually counting the number of 2-bonds in a graph $G$, but instead the number of edges that \emph{can be contained} in a 2-bond in some extension of $G$.

\section{Graphs with 2-bonds and 2-separators}
\label{sec:graphs_with_2_bonds}
Before we define the filtration of $\Ctwobonds$ and prove the connectivity of the inclusions, we need to get a better understanding of the structure of graphs with 2-bonds.

We say that an edge $e\in E(G)$ is \emph{contained in a 2-separator} if there is $G\leq G'\in \FreeS$ such that $e$ is contained in a 2-bond in $G'$.\footnote{In fact, the proof of \cref{lem:equivalent_characterisations_of_E2sep} shows that if such a $G'$ exists, it can be chosen such that $G'$ contains at most one edge more than $G$.} 
We write $\twosepedges(G)\subseteq E(G)$ for the set of edges of $G$ that are contained in a 2-separator.
In other words, for $e\in E(G)$, we have
\begin{gather*}
	e\text{ is contained in a 2-bond in }G \implies e\in \twosepedges(G)\\
	\text{and}\\
	e\in \twosepedges(G) \implies e\text{ is contained in a 2-bond in some } G'\geq G.
\end{gather*}
For $\sigma\in \FreeS$ a sphere system, we let $\twosepedges(\sigma)\subset \sigma$ be the subset corresponding to the edges that lie in a 2-separator in the dual graph,
\begin{equation*}
	\twosepedges(\sigma) = \ls S\in \sigma \mid e_S \in \twosepedges(\Gamma(\sigma)) \rs.
\end{equation*}

The aim of the present section is to provide 2-bond analogs of the results about graphs with 1-bonds in \cref{sec:graphs_with_1_bonds}. The proof ideas are similar, but the technical details get more involved here.
The graph-theoretic input needed later consists mainly of three facts.

\begin{enumerate}
\item \cref{lem:equivalent_characterisations_of_E2sep} is an analogue of \cref{lem:Conebond_via_cut_vertices} and gives a usable characterisation of edges in $\twosepedges$ in terms of vertex--edge cuts.
\item \cref{lem:adding_in_2_sep} is an analogue of \cref{lem:collapsing_not_to_partial_preserves_bridges} and controls how $\twosepedges$ changes when one expands edges in a graph.
\item \cref{lem:collapsing_2_cut_edges,lem:2_sep_collapsed_reduced_size} are related to \cref{lem:bridges_are_a_forest} and show that collapsing only edges in $\twosepedges$ preserves 2-vertex connectivity and properly decreases the size of $\twosepedges$.
\end{enumerate}

Similar graph theoretic results also appear in the work of Bregman--Fullarton, in particular in Section 3.3 \cite{Bregman2018}.

\subsection{Characterising edges in 2-separators}
\label{sec:structure_of_2_bonds}

\cref{lem:Conebond_via_cut_vertices} provided a description of all graphs that are obtained by collapsing edges in a graph with a 1-bond: These are the graphs that lie in $\partial \FreeS$ or contain a cut vertex.
It is almost true that $\sigma\in \Ctwobonds$ if and only if $\sigma \in \Conebonds$ or $\Gamma(\sigma)$ contains a 2-vertex cut. The reason is that collapsing both edges in a 2-bond leads to a 2-vertex cut (see \cref{rem:2_sepators_in_C}).
However, we also need to understand what happens if we just collapse one edge of a 2-bond.
This is described in \cref{lem:equivalent_characterisations_of_E2sep}. Before we state it, we prove two auxiliary lemmas.

\begin{lemma}
	\label{lem:two_bonds_two_comps}
	Let $G\in \FreeS\setminus C$. If $\ls e,f\rs \subset E(G)$ is a 2-bond, then $e$ and $f$ are disjoint, $G-\ls e,f\rs$ has exactly two connected components and each of them contains exactly one endpoint of $e$ and one of $f$.
\end{lemma}
\begin{proof}
	Let $V$ be the set of endpoints of $e$ and $f$. 
	Let $K$ be a connected component of $G-\ls e,f\rs$. 
	By \cref{lem:Conebond_via_cut_vertices}, $G$ is 2-vertex connected, so in particular also 2-edge connected.
	1-connectivity implies that $|K\cap V| \geq 1$; 2-vertex connectivity implies that $|K\cap V| \geq 2$; 2-edge connectivity implies that $K$ cannot contain both endpoints of $e$ or both endpoints of $f$.
	As $G-\ls e,f\rs$ is disconnected and distinct connected components are disjoint, this implies the claim.
\end{proof}

\begin{lemma}
\label{lem:every_2_bond_cyclic}
Let $G\in \FreeS\setminus C$ and $\ls e,f\rs\subset E(G)$ a 2-bond. Then each component of $G-\ls e,f\rs$ has non-trivial fundamental group.
\end{lemma}
\begin{proof}
	By \cref{lem:two_bonds_two_comps}, $G-\ls e,f\rs$ has two connected components. 
	Assume that such a component $K$ has trivial fundamental group.
	By \cref{lem:two_bonds_two_comps}, $K$ contains exactly one endpoint of $e$ and one endpoint of $f$ and these are distinct. Every other vertex of $K$ must have valence at least three. This implies that $K$ can only be a tree if it consists of a single edge. But this cannot be the case as again by \cref{lem:two_bonds_two_comps}, the endpoints of $e$ and $f$ are disjoint, i.e.~if $K$ would be a single edge, there would be vertices of valence two.
\end{proof}

The following lemma gives a more concrete description of $\twosepedges(\sigma)$ that we will work with in the following.

\begin{lemma}
	\label{lem:equivalent_characterisations_of_E2sep}
	Let $G\in \FreeS\setminus C$ and $e\in E(G)$. Then $e\in \twosepedges(G)$ if and only if there is $v\in V(G)$ such that $|G|- \ls v,e\rs$ is disconnected.

	If this is the case, then one can choose $v$ such that $|G|- \ls v,e\rs$ has exactly two connected components, each containing an endpoint of $e$ as a proper subset.
\end{lemma}

\begin{proof}
	Assume that $e\in \twosepedges(G)$.
	If $e$ is contained in a 2-bond $\ls e,f\rs$ in $G$, then by \cref{lem:two_bonds_two_comps}, $G- \ls e,f\rs$ has exactly two connected components $K_1$ and $K_2$, each containing an endpoint of $e$ and that endpoint is a proper subset by \cref{lem:every_2_bond_cyclic}. 
	Let $v$ be the endpoint of $f$ contained in $K_1$. Then $|G|- \ls v, e\rs$ also has two connected components: One is equal to $K_1\setminus \ls v \rs$; this is connected because the assumptions imply that if $v$ was a cut vertex in $K_1$, then it would be a cut vertex in $G$, contradicting the assumption that $G\in \FreeS\setminus C$. The other connected component is obtained from $K_2$ by adding the open edge $f$. Hence, the claim follows.
	
	Now assume that $e$ is not contained in a 2-bond in $G$.
	By definition, there is $G\leq G''\in \FreeS$ such that $e$ is contained in a 2-bond $\ls e,f\rs$ in $G''$. Let $G\leq G'\leq G''$ be obtained by collapsing all edges in $E(G'')\setminus E(G)$ other than $e,f$. It is easy to see that $\ls e,f\rs$ is still a 2-bond in $G'$.
	Furthermore, we have $G'\in \FreeS\setminus \Conebonds$ by \cref{lem:rk_preserved_2_connected} as $G\leq G'$.
	So $\ls e,f\rs$ is a 2-bond in $G'$ and $G$ is obtained from $G'$ by collapsing $f$ to a vertex $w$.
	By \cref{lem:two_bonds_two_comps} and \cref{lem:every_2_bond_cyclic}, $G'-\ls e,f\rs$ has exactly two connected components, each of which has non-trivial fundamental group. In particular, each connected component contains more than one vertex. But then also $|G'|- \ls e, \bar{f}\rs$ is disconnected, where $\bar{f}$ is the closure of $f$ in $|G'|$, i.e. the edge together with its endpoints.
	As $f$ gets collapsed to $w$ in $G$, we have that $|G'|- \ls e, \bar{f}\rs$ is homeomorphic to $|G|- \ls e, w\rs$, so setting $v = w$, the claim follows. 

	For the converse, assume that we have $v$ such that $|G|- \ls v,e\rs$ is disconnected. Note that this implies that $v$ is not an endpoint of $e$ because otherwise, it would be a cut vertex in $G$; as $G\in \FreeS\setminus \Conebonds$, \cref{lem:Conebond_via_cut_vertices} says that this is not possible.
	We claim that we can blow up $v$ to an edge $f$ such that $\ls e, f\rs$ is a 2-bond in the resulting graph $G'$. This implicitly uses the poset of partitions description of the link of a vertex in $\FreeS$ that we discussed in \cref{sec:link_sphere_system}.
	Let $v_1, v_2$ be the two endpoints of $e$. We can partition the set of edges of $G$ that are incident to $v$ into two sets $E_1$ and $E_2$ such that all edges in $E_1$ lie in the same connected component of $|G|- \ls v,e \rs$ as $v_1$ and all edges in $E_2$ lie in the same connected component of $|G|- \ls v,e \rs$ as $v_2$.
	Both $E_1$ and $E_2$ must have cardinality at least one because otherwise, the corresponding endpoint of $e$ would be a cut vertex in $G$.
	If both $E_1$ and $E_2$ have cardinality at least two, we blow up $v$ to an edge $f$ such that all edges in $E_1$ are incident to one endpoint of $f$ and all edges in $E_2$ are incident to the other endpoint of $f$. As $|E_1|, |E_2|\geq 2$, every vertex of the resulting graph $G'$ still has valence at least three, so $G'\in \FreeS$. Furthermore, the edges in $E_1$ lie in the same connected component of $|G'|- \ls e,f\rs$ as $v_1$ and the edges in $E_2$ lie in the same connected component of $|G'|- \ls e,f\rs$ as $v_2$. So $\ls e,f\rs$ is a 2-bond in $G'$.
	So assume that without loss of generality, $E_1= \ls e_1 \rs$ has cardinality one. 
	As $v_1$ has valence at least three, the component of $|G|- \ls v,e\rs$ that contains $v_1$ must contain at least one edge in addition to $e_1$. As $e_1$ is the only edge in this component that is adjacent to $v$, it follows that all of these edges lie in a connected component of $G- \ls e_1,e\rs$ that is different from the component containing $v$. Hence, $\ls e, e_1\rs$ is a 2-bond in $G$, see \cref{fig:equivalent_characterisations_of_E2sep}.
	\begin{figure}
		\centering
		\includegraphics{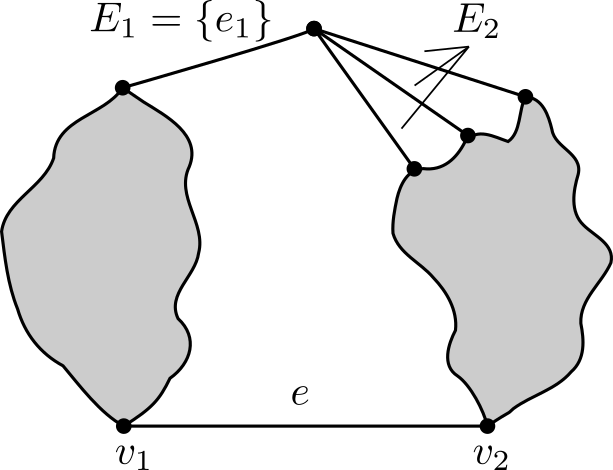}
		\caption{If $E_1$ is a singleton, then $\ls e, e_1\rs$ is a 2-bond in $G$.}
		\label{fig:equivalent_characterisations_of_E2sep}
	\end{figure}
\end{proof}

\begin{remark}
	\label{rem:2_sepators_in_C}
	Similar to the argument in \cref{lem:equivalent_characterisations_of_E2sep}, one can show that every graph in $\FreeS\setminus \Conebonds$ contains a 2-vertex cut.
	The converse is almost true: If $G\in \FreeS\setminus \Conebonds$ contains a 2-vertex cut, then $G\in \Ctwobonds$ except if $n=2$ and $G$ is the $\theta$-graph, i.e.~the graph with two vertices $v,w$ and three non-loop edges. The set $\ls v,w\rs$ is a 2-vertex cut as $|G|-\ls v,w \rs$ consists of three connected components, each given by one of the open edges. However, there is no 2-bond in $G$ and it also does not lie below a graph with a 2-bond because it is a maximal element in $\FreeS$.
	That this is the only exception will get clear in \cref{sec:partitions}, where we describe when one can introduce a 2-bond at a 2-vertex cut.
\end{remark}

\subsection{Behaviour under edge expansion and collapse}
Below, we will use the size of the set $\twosepedges(\sigma)$ as a measure of how far $\Gamma(\sigma)$ is from being 3-vertex connected, and we will use it to define a filtration of $\Ctwobonds$ (see \cref{sec:intuition}).
Hence, we need to understand how this set behaves when we add elements to or remove elements from $\sigma$ (i.e., when we collapse or expand edges in $\Gamma(\sigma)$).
In the setting of 1-bonds, we used \cref{lem:collapsing_not_to_partial_preserves_bridges}, which said that for $\rho, \sigma \in \FreeS\setminus \partial \FreeS$ with $\rho\subseteq \sigma$, we have $\onesepedges(\rho) = \onesepedges(\sigma) \cap \rho$. 
An analogue of one of the inclusions holds true for 2-bonds as well:

\begin{lemma}
	\label{lem:collapsing_not_to_C_preserves_2_cut_edges}
	Let $\rho, \sigma\in \FreeS\setminus C$ with $\rho \subseteq \sigma$.
	Then $\twosepedges(\sigma)\cap \rho \subseteq \twosepedges(\rho)$.
\end{lemma}
\begin{proof}
	This follows immediately from the definition: If $S\in \twosepedges(\sigma)$, then there is $\tau \supseteq \sigma$ where $S$ lies in a 2-bond. But then $\tau\supseteq\rho$, so if $S\in \rho$, we get $S\in \twosepedges(\rho)$.
\end{proof}

The other inclusion, however, does not hold in general for 2-bonds. It can happen that removing elements from $\sigma$ creates new elements in $\twosepedges$, so there can be $\rho\subset \sigma$ such that $|\twosepedges(\rho)|>|\twosepedges(\sigma)|$. An example is given in \cref{fig:new_2_sep_edge}.
\begin{figure}
\centering
\includegraphics{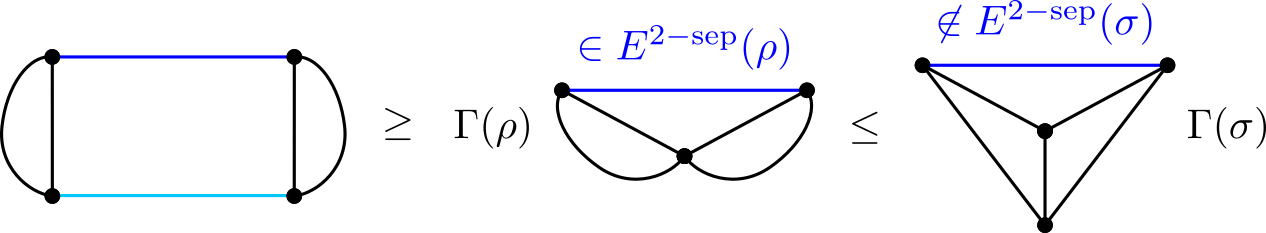}
\caption{The dual graphs of $\rho\subseteq \sigma$ with $\rho,\sigma \in \FreeS\setminus C$ (middle and right). The blue edge is contained in a 2-separator in $\Gamma(\rho)$ (as witnessed by the graph on the left) but not in $\Gamma(\sigma)$. In particular, $\twosepedges(\rho)\not\subseteq \twosepedges(\sigma)$.}
\label{fig:new_2_sep_edge}
\end{figure}
However, a weaker statement is still true and enough for our purposes: If one removes elements from $\sigma$ that are themselves contained in 2-separators, then one does not create new elements that are contained in 2-separators.
This is one of the consequences (see \cref{lem:2_sep_collapsed_reduced_size} in the next subsection) of the following lemma.

\begin{lemma}
	\label{lem:adding_in_2_sep}
	Let $\sigma\in \FreeS\setminus C$. Let $S\in \twosepedges(\sigma)$ and $T$ such that $\sigma\cup \ls T \rs \in \FreeS$ and such that $T\in \twosepedges(\rho)$ for some $\ls S \rs \subset \rho \subseteq \sigma\cup \ls T \rs$ with $\rho\in \FreeS \setminus \Conebonds$.
	Then $S \in \twosepedges(\sigma\cup \ls T \rs)$.
\end{lemma}
\begin{proof}
	Let $G = \Gamma(\sigma\cup \ls T \rs)$ and $H = \Gamma(\sigma)$. We have $H = G/f$, where $f\in E(G)$ is the edge corresponding to $T$. Let $v_f\in V(H)$ be the vertex that $f$ collapses to.
	Let $e\in \twosepedges(H)$ be the edge corresponding to $S$. We need to show that $e\in \twosepedges(G)$.

	Let $v_1, v_2$ and $w_1, w_2$ be the endpoints of $e$ and $f$, respectively.
	By \cref{lem:equivalent_characterisations_of_E2sep}, there is a vertex $v\in V(H)$ such that 
	$|H|-\ls v,e\rs$ has exactly two components $K_1, K_2$, one containing $v_1$ and the other containing $v_2$.
	If $v\neq v_f$, then it is easy to see that $v_1$ and $v_2$ also lie in distinct components of $|G|- \ls v, e\rs$, so $e\in \twosepedges(G)$ by the same \cref{lem:equivalent_characterisations_of_E2sep}.

	Now assume that $v= v_f$. 
	If $v_1$ and $v_2$ lie in distinct components of $|G|-\ls w_1, e\rs$ or $|G|-\ls w_2, e\rs$, then by \cref{lem:equivalent_characterisations_of_E2sep}, we get that $e\in \twosepedges(G)$. 
	So assume that this is not the case. Then there is a path $P_{\neg 1}$ from $v_1$ to $v_2$ in $|G|- \ls w_1, e \rs$. This path avoids $w_1$, so it cannot contain the open edge $f$, but it can contain its endpoint $w_2$. Write $P_{\neg 1}$ as a concatenation $P_{\neg 1}^1 w_2 P_{\neg 1}^2$, where both $P_{\neg 1}^1$ and $P_{\neg 1}^2$ do not contain $w_2$ and hence are paths in $|G|- \ls \overbar{f}, e \rs$.
	Since $v_f$ is the vertex that $f$ collapses to, there is a homeomorphism 
	\begin{equation*}
		K_1 \sqcup K_2 = |H|-\ls v_f,e\rs \cong |G|- \ls \overbar{f}, e \rs.
	\end{equation*}
	It follows that $P_{\neg 1}^1$ is contained in $K_1$ and $P_{\neg 1}^2$ is contained in $K_2$, see \cref{fig:3_15_replacement}.
	Similarly, there is a path $P_{\neg 2}$ from $v_2$ to $v_1$ in $|G|- \ls e, w_2 \rs$ that has an initial part $P_{\neg 2}^2 \subseteq K_2$ that ends at $w_1$ and a terminal part $P_{\neg 2}^1 \subseteq K_1$ that starts at $w_1$. 
	\begin{figure}
		\centering
		\includegraphics{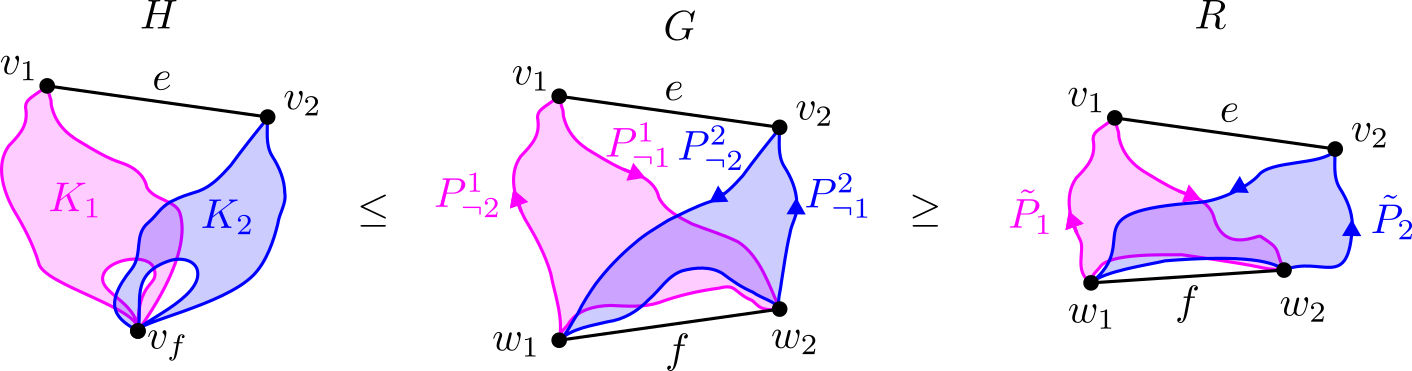}
		\caption{Paths and connected components in graphs of the proof of \cref{lem:adding_in_2_sep}.}
		\label{fig:3_15_replacement}
	\end{figure}
	The concatenations $P_1\coloneqq P_{\neg 2}^1  P_{\neg 1}^1$ and $P_2\coloneqq P_{\neg 1}^2 P_{\neg 2}^2$ give two paths in $G-\ls f\rs$ connecting $w_1$ to $w_2$. The interiors of $P_1$ and $P_2$ are disjoint because the former is contained in $K_1$ and the latter is contained in $K_2$.
	By assumption, there is $\rho \subseteq \sigma\cup \ls T\rs$ such that $S\in \rho\in \FreeS\setminus \Conebonds$ and $T\in \twosepedges(\rho)$.
	The graph $R$ corresponding to $\rho$ is obtained from $G$ by collapsing a set of edges not containing $e$ or $f$. It follows that $P_1$ and $P_2$ induce paths $\tilde{P_1}$ and $\tilde{P_2}$ in $R-\ls f\rs$ that connect the endpoints of $f$ and are disjoint in their interiors.
	This implies that for any vertex $u\in V(R)$ different from the endpoints of $f$, there is a path from $w_1$ to $w_2$ in $|R|-\ls u,f \rs$.
	By \cref{lem:equivalent_characterisations_of_E2sep}, this contradicts the assumption that $f\in \twosepedges(R)$.
\end{proof}

We will need the following consequence of \cref{lem:adding_in_2_sep}.

\begin{lemma}
	\label{lem:2_sep_pairwise}
	Let $\rho\subset \sigma \in \FreeS$. Assume that for all $S,T\in \sigma\setminus \rho$, we have $S,T\in \twosepedges(\rho\cup \ls S,T \rs)$. Then  $\sigma\setminus \rho \subseteq \twosepedges(\sigma)$.
\end{lemma}
\begin{proof}
	We prove by induction that for all $\rho\subseteq\mu\subseteq\sigma$ such that $|\mu\setminus \rho| = i$, we have $\mu\setminus \rho \subseteq \twosepedges(\mu)$.
	For $i= 2$, this is true by assumption and for $i=1$, it follows from the assumption using \cref{lem:collapsing_not_to_C_preserves_2_cut_edges}. Now assume that the claim holds for some $i\geq 2$, and let $\rho\subseteq\mu\subseteq\sigma$ be such that $\mu\setminus \rho = \ls S_1,\ldots, S_{i+1} \rs$ has size $i+1$.

	Let $1\leq j \leq i+1$. To complete the induction step, we need to show that $S_j\in \twosepedges(\mu)$.
	Let $1\leq k\leq i+1$ with $k\neq j$.
	By the induction hypothesis, we have $S_j\in \twosepedges(\mu\setminus \ls S_k\rs)$ and we also have $S_k\in \twosepedges(\rho\cup \ls S_j, S_k \rs)$. Hence, we can apply \cref{lem:adding_in_2_sep} and get that $S_j\in \twosepedges(\mu)$, where we use that $\mu = \left(\mu\setminus \ls S_k \rs\right) \cup \ls S_k \rs$.
\end{proof}

\subsection{The 2-bonds form a forest}

The last lemma that we want to transport to the setting of 2-bonds is \cref{lem:bridges_are_a_forest}, which says that the separating edges of each graph in $\FreeS\setminus \partial \FreeS$ form a forest.
In the setting of 2-bonds, we have the following lemma, which is slightly stronger (see \cref{cor:2_sep_edges_form_a_forest}).

\begin{lemma}
	\label{lem:collapsing_2_cut_edges}
	Let $\sigma \in \FreeS\setminus C$ and $\rho\subseteq \sigma$ with $\sigma \setminus \rho\subseteq \twosepedges(\sigma)$. Then $\rho \in \FreeS\setminus C$.
\end{lemma}

\begin{proof}
	Write $\sigma\setminus \rho = \{S_1,\ldots,S_k\}\subseteq \twosepedges(\sigma)$ and set $\rho_i \coloneqq \sigma\setminus \ls S_1,\ldots,S_{i}\rs$ for each $0\leq i \leq k$. We have a chain of inclusions
	\begin{equation*}
		\sigma = \rho_0 \supset \rho_1 \supset \cdots \supset \rho_k = \rho
	\end{equation*}
	and will iteratively show that for all $i$, we have $\rho_i \in \FreeS \setminus \Conebonds$. This is by definition true for $\rho_{0}=\sigma$. Now assume that we have shown that $\rho_{i-1} \in \FreeS \setminus C$. Then we get
	\begin{equation*}
		\rho_{i-1}\setminus \rho_{i} = \ls S_i\rs \subseteq \twosepedges(\sigma)\cap \rho_{i-1}\subseteq \twosepedges(\rho_{i-1}),
	\end{equation*}
	where the last inclusion follows from \cref{lem:collapsing_not_to_C_preserves_2_cut_edges}. Hence, replacing $\sigma$ by $\rho_{i-1}$ and $\rho$ by $\rho_i$, it is enough to prove the lemma under the assumption that $|\sigma\setminus \rho| = 1$.

	To prove this, let $G = \Gamma(\sigma)$ and $H = \Gamma(\rho)$. We have $H = G/e$, where $e\in \twosepedges(G)$. It is clear that $H\in \FreeS$, so we need to show that $H\not\in C$.
	For this, we use the characterisation given by \cref{lem:Conebond_via_cut_vertices}.
	First observe that as $G\in \FreeS\setminus \Conebonds$, it has no cut vertex by \cref{lem:Conebond_via_cut_vertices}. In particular, the edge $e$ is not a loop. This implies that $\rk(H) = \rk(G)$, so $H\not \in \partial \FreeS$.

	It remains to show that $H$ has no cut vertex. That it is connected is clear. It is also clear that no vertex other than $v_e\in V(G/e)$, the one that $e$ collapses to, can be a cut vertex.
	So we need to show that $|G/e| - \ls v_e\rs$, which is homeomorphic to $|G|- \ls \bar{e} \rs$, is connected.
	As $e\in \twosepedges(G)$, \cref{lem:equivalent_characterisations_of_E2sep} says that there is a vertex $v\in V(G)$ such that the endpoints $v_1, v_2$ of $e$ lie in distinct components of $|G|- \ls v, e\rs$.
	Let $w$ be any point on $|G|- \ls \bar{e} \rs$.
	We will show that $w$ lies in the same component of $|G|- \ls \bar{e} \rs$ as $v$.
	As neither $v_1$ nor $v_2$ are cut vertices, there is a path $P_{\neg 1}$ in $G$ from $w$ to $v$ that avoids $v_1$ and a path $P_{\neg 2}$ in $G$ from $w$ to $v$ that avoids $v_2$. If either $P_{\neg 1}$ or $P_{\neg 2}$ avoids both endpoints of $e$, it gives a path from $w$ to $v$ in $|G|- \ls \bar{e} \rs$. Otherwise, $P_{\neg 1}$ contains $v_2$ and $P_{\neg 2}$ contains $v_1$. Their initial segments then give a path $P$ in $G-e$ between the endpoints of $e$ that avoids $v$ (see \cref{fig:collapse_2_cut_no_separator}). This contradicts the assumption that the endpoints lie in different components of $|G|- \ls v, e\rs$.
	\begin{figure}
		\centering
		\includegraphics{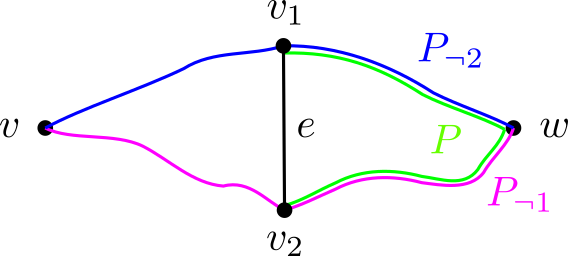}
		\caption{The paths $P_{\neg 1}$, $P_{\neg 2}$ and $P$ in the proof of \cref{lem:collapsing_2_cut_edges}.}
		\label{fig:collapse_2_cut_no_separator}
	\end{figure}
\end{proof}

\begin{corollary}
	\label{cor:2_sep_edges_form_a_forest}
	If $\sigma \in \FreeS\setminus C$, then $\twosepedges(\sigma)$ forms a forest inside $\Gamma(\sigma)$.
\end{corollary}
\begin{proof}
	Let $\rho = \sigma\setminus \twosepedges(\sigma)$. By \cref{lem:collapsing_2_cut_edges}, we have $\rho\in \FreeS\setminus \Conebonds$, so in particular $\rho\not\in \partial \FreeS$. This means that $\rk(\Gamma(\sigma)) = \rk(\Gamma(\rho))$. As $\Gamma(\rho)$ is obtained from $\Gamma(\sigma)$ by collapsing all edges in $\twosepedges(\sigma)$, it follows that $\twosepedges(\sigma)$ forms a forest.	
\end{proof}

Using \cref{lem:collapsing_2_cut_edges}, we get another consequence of \cref{lem:adding_in_2_sep} that we will need later:

\begin{lemma}
	\label{lem:2_sep_collapsed_reduced_size}
	Let $\sigma\in \FreeS\setminus C$ such that $\rho\subset \sigma$ with $\sigma \setminus \rho \subseteq \twosepedges(\sigma)$. Then $\twosepedges(\rho)\subset \twosepedges(\sigma)$.
\end{lemma}
\begin{proof}
	Write $\sigma\setminus \rho = \{S_1,\ldots,S_k\}\subseteq \twosepedges(\sigma)$ and for each $0\leq i \leq k$, set $\rho_i \coloneqq \rho\cup \ls S_1,\ldots,S_{i}\rs$.
	We have a chain of inclusions
	\begin{equation*}
		\rho = \rho_0 \subset \rho_1 \subset \cdots \subset \rho_k = \sigma.
	\end{equation*}
	By \cref{lem:collapsing_2_cut_edges}, we have $\rho_i \in \FreeS\setminus C$ for all $i$.
	Let $S\in \twosepedges(\rho)$.
	We prove by induction that $S\in \twosepedges(\rho_i)$ for all $i$.
	For $i=0$, this is true by definition. Now assume that $S \in \twosepedges(\rho_i)$ for some $i$.
	Set $T = S_{i+1}$. 
	Then $\rho_i\cup \ls T\rs = \rho_{i+1} \in \FreeS$.
	We have	$T\in \twosepedges(\sigma)$, so by \cref{lem:collapsing_not_to_C_preserves_2_cut_edges}, we also have $T\in \twosepedges(\rho_{i+1})$. 
	Hence, we can apply \cref{lem:adding_in_2_sep} and get that $S\in \twosepedges(\rho_{i+1})$.
\end{proof}

\section{A filtration of $\Conebonds\hookrightarrow \Ctwobonds$ and its fibres}

\subsection{The filtration}

We are now ready to define the filtration of $\Ctwobonds$ that we will use to prove \cref{thm:C'_he_C}; it is the 2-bond analogue of the filtration in \cref{def:Conebonds}.

\begin{definition}
\label{def:Ctwobonds}
Let $\Ctwobonds_{p,q}$ be the subposet of $\Ctwobonds$ consisting of all $\sigma$ such that 
\begin{enumerate}
\item \label{it:contained_in_C}$\sigma\in C$ or
\item \label{it:more_bonds}$|\twosepedges(\sigma)|> p$ or
\item \label{it:less_non_bonds}$|\twosepedges(\sigma)| = p$ and $|\sigma| \leq p+q$.
\end{enumerate}
\end{definition}

In words, the poset $\Ctwobonds_{p,q}$ consists of all $\sigma$ that either lie in $C$, contain more than $p$ spheres in 2-separators or contain exactly $p$ spheres in 2-separators but at most $q$ spheres not contained in 2-separators.

\begin{lemma}
	\label{lem:filtration_of_Ctwobonds}
	\begin{enumerate}
	\item \label{it:limit_size_two_sep}If $p> 2n-3$, then $\Ctwobonds_{p,q} = \Conebonds$.
	\item \label{it:small_q}If $q\leq n-1$, we have $\Ctwobonds_{p,q} = \Ctwobonds_{p+1,3n-4-p}$ and for $p=0$, even $\Ctwobonds_{0,n-1} = \Ctwobonds_{0,n}$.
	\item \label{it:zero_p}If $q\geq 3n-5$, we have $\Ctwobonds_{0,q} = \Ctwobonds$.
\end{enumerate}
\end{lemma}
\begin{proof}
	Similarly to the proof of \cref{lem:filtration_of_Conebonds}, we use that $|\sigma|\leq 3n-3$ for all $\sigma \in \FreeS$.

	For $\sigma \in \Ctwobonds\setminus \Conebonds$, \cref{cor:2_sep_edges_form_a_forest} says that the edges in $\twosepedges(\sigma)$ form a forest.
	This implies that 
	\begin{equation*}
		|\twosepedges(\sigma)| \leq |V(\Gamma(\sigma))| - 1 \leq 2n-3 \text{ and } |\sigma| - |\twosepedges(\sigma)| \geq n.
	\end{equation*}
	This implies \cref{it:limit_size_two_sep} and  \cref{it:small_q} for $p>0$. The sharper version of \cref{it:small_q} for $p=0$ follows from the fact that if $\sigma$ has $n$ edges and does not lie in $\Conebonds$, then $\sigma$ must be a rose and as this has a cut vertex, it is not contained in $\Ctwobonds$.

	To obtain \cref{it:zero_p}, we use that if $\sigma\in \Ctwobonds\setminus \Conebonds$ and $\twosepedges(\sigma) = \emptyset$, then in $\Gamma(\sigma)$, there needs to be a 2-separator (obtained as the image of a collapsed 2-bond, see \cref{rem:2_sepators_in_C}) and these two vertices need to both have valence at least four.
\end{proof}

Using this lemma, we get the following filtration of $\Ctwobonds$:
\begin{multline*}
\Conebonds = \Ctwobonds_{2n-2,\bullet} = \Ctwobonds_{2n-3,n-1} \subseteq \Ctwobonds_{2n-3,n} =  \Ctwobonds_{2n-4,n-1} \subseteq \Ctwobonds_{2n-4,n} \subseteq
\Ctwobonds_{2n-4,n+1} = \\
= \Ctwobonds_{2n-5,n-1} \subseteq \cdots \subseteq \Ctwobonds_{1,3n-4} = \Ctwobonds_{0,n} \subseteq \cdots \subseteq \Ctwobonds_{0,3n-5} = \Ctwobonds.
\end{multline*}

To prove \cref{thm:C'_he_C}, we will show that the inclusion $\Ctwobonds_{0,n}\hookrightarrow \Ctwobonds_{0,n+1}$ is $(2n-3)$-connected and all other inclusions in this filtration are homotopy equivalences.

This makes precise the intuition that the filtration level of $\sigma$ is determined by ``how many 2-bonds it contains'' (see \cref{sec:intuition}): We order the graphs occurring in $\Ctwobonds\setminus \Conebonds$ in such a way that we first add all graphs that have (at least) $2n-3$ edges that are contained in a 2-separator, then all graphs that have (at least) $2n-4$ such edges, and so on. 
We then refine this ordering by saying that among all graphs that have at least $p$ edges in a 2-separator, we first add those with the least number of edges in total, then those with one more edge, and so on.

\subsection{The fibres}

To show that $f\colon\Ctwobonds_{p,q-1}\hookrightarrow \Ctwobonds_{p,q}$ is highly connected, we show that the fibres are all highly connected and then apply Quillen's fibre theorem as stated, see \cref{thm:Quillen_fibre}.
Regarding the fibres, we will show the following two results:

\begin{proposition}
	\label{prop:lower_fibre}
	If $p\geq 1$, then for all $\tau\in \Ctwobonds_{p,q}\setminus \Ctwobonds_{p,q-1}$, the lower fibre $f^{-1}((\Ctwobonds_{p,q})_{\subseteq \tau})$ is contractible.
\end{proposition}

\begin{proposition}
	\label{prop:upper_fibre}
	If $p=0$, then for all $\rho\in \Ctwobonds_{p,q}\setminus \Ctwobonds_{p,q-1}$, the upper interval $(\Ctwobonds_{p,q})_{\supset \rho}$ is contractible except if $\Gamma(\rho)$ is a $\theta$-graph. In this case, $(\Ctwobonds_{p,q})_{\supset \rho}$ is homotopy equivalent to a wedge of $(n-1)!$ spheres of dimension $n-3$ and $f^{-1}((\Ctwobonds_{p,q})_{\subseteq \rho})$ is isomorphic to the boundary of an $n$-simplex.
\end{proposition}

\subsection{The lower fibre}
\label{sec:lower_fibre}
In this section, we prove \cref{prop:lower_fibre}. We fix $p$ (for now we do not assume $p\geq 1$), $q$ and $\tau\in \Ctwobonds_{p,q}\setminus \Ctwobonds_{p,q-1}$. Let $f_{\subset\tau}\coloneqq f^{-1}((\Ctwobonds_{p,q})_{\subseteq \tau})$. 

\begin{lemma}
	\label{lem:characterisation_lower_fibre}
	We have $\tau\in \FreeS\setminus \Conebonds$, $|\twosepedges(\tau)|=p$, $|\tau| = p+q$ and 
	\begin{equation*}
		f_{\subset\tau} = \ls \emptyset\neq \sigma\subset \tau \,\middle|\, \sigma\in \Conebonds \text{ or } |\twosepedges(\sigma)| \geq p\rs.
	\end{equation*}
\end{lemma}
\begin{proof}
	The statements about $\tau$ follow immediately from the definition of $\Ctwobonds_{p,q}$.
	For the characterisation of the lower fibre, note that $\sigma \in f_{\subset\tau}$ if and only if $\emptyset\neq \sigma\subset \tau$ and $\sigma\in \Ctwobonds_{p,q-1}$. Using \cref{def:Ctwobonds}, the latter means that $\sigma\in \Conebonds$, $|\twosepedges(\sigma)| > p$ or $|\twosepedges(\sigma)| = p$ and $|\sigma| \leq p+q-1$. The condition $|\sigma| \leq p+q-1$ is vacuous for $\sigma\subset \tau$ as $|\tau| = p+q$.
\end{proof}

\subsubsection{Alternative model for the lower fibre}

Recall from \cref{sec:intuition} that the lower fibre of $\tau$ should intuitively consist of all $\sigma$ such that $\Gamma(\sigma)$ does not have ``less 2-bonds'' than $\Gamma(\tau)$, i.e.~is obtained by only collapsing edges that do not lie in 2-separators. That the fibre looks like this is not true right away (see \cref{lem:characterisation_lower_fibre}), but we will show now that this intuition is correct up to homotopy equivalence.

Let $\fibrenontwo$ be the subposet of $\FreeS_{\subset\tau}$ consisting of those $\emptyset\neq\sigma\subset \tau$ that are obtained by only removing spheres that do not lie in $\twosepedges(\tau)$,
\begin{equation*}
	\fibrenontwo = \ls \emptyset\neq \sigma\subset \tau \mid \twosepedges(\tau) \subseteq \sigma \rs.
\end{equation*}
It is easy to see that this is a subposet of the lower fibre:

\begin{lemma}
	\label{lem:fibrenontwo_subseteq_lower_fibre}
	We have $\fibrenontwo\subseteq f_{\subset\tau}$.
\end{lemma}
\begin{proof}
	Let $\sigma\in \fibrenontwo$. By definition, we have $\emptyset\neq \sigma\subset \tau$. By \cref{lem:characterisation_lower_fibre}, we need to show that $\sigma \in \Conebonds$ or $|\twosepedges(\sigma)| \geq p$.
	So assume that $\sigma\in \FreeS \setminus \Conebonds$. By \cref{lem:rk_preserved_2_connected}, we have $\tau \in \FreeS\setminus C$.
	Then by \cref{lem:collapsing_not_to_C_preserves_2_cut_edges}, we have $\twosepedges(\tau)\subseteq \twosepedges(\sigma)$, so in particular $|\twosepedges(\sigma)|\geq |\twosepedges(\tau)| = p$.
\end{proof}

We next observe that when passing to any $\sigma\in f_{\subset\tau}$, we need to collapse at least one edge that does not lie in a 2-separator:

\begin{lemma}
\label{lem:collapsed_things_outside_2_bond}
If $\sigma \in f_{\subset\tau}$, then $\tau\setminus \sigma \not\subseteq \twosepedges(\tau)$.
\end{lemma}
\begin{proof}
	We prove the contrapositive: Assume that $\sigma\subset \tau$ with $\tau\setminus \sigma \subseteq \twosepedges(\tau)$.
	Then \cref{lem:collapsing_2_cut_edges} implies that $\sigma\not \in \Conebonds$ because $\tau\not\in \Conebonds$ and \cref{lem:2_sep_collapsed_reduced_size} implies that $|\twosepedges(\sigma)|<|\twosepedges(\tau)| = p$.
	Now \cref{lem:characterisation_lower_fibre} implies that $\sigma\not \in f_{\subset\tau}$.
\end{proof}

Using the previous two lemmas, we can show that the lower fibre is homotopy equivalent to $\fibrenontwo$.

\begin{lemma}
	\label{prop:lower_fibre_simeq_fibrenontwo}
	We have $f_{\subset\tau} \simeq \fibrenontwo$.
\end{lemma}
\begin{proof}
	By \cref{lem:fibrenontwo_subseteq_lower_fibre}, we know that $\fibrenontwo$ is a subposet of $f_{\subset\tau}$.
	We define a monotone poset map $f_{\subset\tau} \to \fibrenontwo$ that restricts to the identity on $\fibrenontwo$ as follows:
	For $\sigma\in f_{\subset\tau}$, let $\sigma\subseteq \sigma' \subseteq \tau$ be defined by $\sigma' = \twosepedges(\tau)\cup \sigma$.
	To see that $\sigma'\in \fibrenontwo$, it is enough to note that $\sigma'\neq\tau$ by \cref{lem:collapsed_things_outside_2_bond}. 
	That the assignment $\sigma\mapsto\sigma'$ is monotone and restricts to the identity on $\fibrenontwo$ is clear.
\end{proof}

\subsubsection{Contractibility of the lower fibre}
We now assume that $p\geq 1$ and prove \cref{prop:lower_fibre}.

\begin{lemma}
	\label{lem:fibrenontwo_contractible}
	If $p\geq 1$, then the poset $\fibrenontwo$ is contractible.
\end{lemma}
\begin{proof}
	Let $\hat \tau \coloneqq \twosepedges(\tau)$. 
	As $p\geq 1$, we have $\emptyset\neq \hat{\tau}$.
	By \cref{cor:2_sep_edges_form_a_forest}, $\twosepedges(\tau)$ forms a forest in $\Gamma(\tau)$, so we have $\hat{\tau}\subset \tau$.
	It follows that $\hat{\tau}\in\fibrenontwo$. By definition, every $\sigma\in\fibrenontwo$ satisfies $\hat{\tau}\subseteq\sigma $. So $\hat{\tau}$ forms a unique minimal element in $\fibrenontwo$ and gives a cone point.
\end{proof}

\begin{proof}[Proof of \cref{prop:lower_fibre}]
	The claim of \cref{prop:lower_fibre} follows immediately from \cref{prop:lower_fibre_simeq_fibrenontwo} and \cref{lem:fibrenontwo_contractible}.
\end{proof}

Using \cref{prop:lower_fibre} and \cref{thm:Quillen_fibre}, our filtration can be rewritten as
\begin{equation*}
\Conebonds \hookrightarrow \Ctwobonds_{1,3n-4} = \Ctwobonds_{0,n} \subset \Ctwobonds_{0,n+1} \subset \cdots \subset \Ctwobonds_{0,3n-5} = \Ctwobonds,
\end{equation*}
where the first inclusion is a homotopy equivalence.
We will now consider the upper intervals to study the remaining inclusions.
We will see that $\Ctwobonds_{0,n}\subset \Ctwobonds_{0,n+1}$ is $(2n-3)$-connected and that all remaining inclusions are homotopy equivalences.

\subsection{The upper interval}
In this section, we start working towards \cref{prop:upper_fibre}.
We fix $p$ (for now we do not assume $p=0$), $q$ and $\rho\in \Ctwobonds_{p,q}\setminus \Ctwobonds_{p,q-1}$. Let $\fibrealltwo\coloneqq (\Ctwobonds_{p,q})_{\supset \rho}$.

\begin{lemma}
	\label{lem:characterisation_upper_fibre}
	We have $\rho\not\in \Conebonds$, $|\twosepedges(\rho)|=p$, $|\rho| = p+q$ and
	\begin{equation*}
		\fibrealltwo = \ls \rho \subset \tau\in \FreeS \,\middle|\, |\twosepedges(\tau)|>p \rs.
\end{equation*}
\end{lemma}
\begin{proof}
	The statements about $\rho$ follow immediately from the definition of $\Ctwobonds_{p,q}$.
	By definition, we have $\tau\in \fibrealltwo$ if and only if $\tau\in \FreeS$, $\rho\subset \tau$ and $\tau\in \Ctwobonds_{p,q}$. 
	If $\rho\subset \tau$, we clearly have $|\tau| > |\rho| = p+q$ and by \cref{lem:rk_preserved_2_connected}, we also have $\tau\in \FreeS\setminus C$. Hence by \cref{def:Ctwobonds}, we have $\tau\in \Ctwobonds_{p,q}$ if and only if $|\twosepedges(\tau)|>p$.
\end{proof}

\subsubsection{Alternative model for the upper interval}
Recall from \cref{sec:intuition} that the upper interval of $\rho$ should intuitively consist of all $\sigma$ such that $\Gamma(\sigma)$ is obtained from $\Gamma(\tau)$ by ``adding 2-bonds'' (or rather edges in 2-separators). Again, that the interval looks like this is not true right away (see \cref{lem:characterisation_upper_fibre}), but we will show that this intuition is correct up to homotopy equivalence.

Let $\fibreonlytwo$ be the subposet of $\FreeS_{\supset\rho}$ consisting of those $\tau\supset \rho$ from which $\rho$ is obtained by only removing spheres that lie in a 2-separator (of $\tau$),
\begin{equation*}
	\fibreonlytwo = \ls \rho \subset \tau\in \FreeS \,\middle|\, \tau\setminus\rho \subseteq \twosepedges(\tau) \rs.
\end{equation*}
This is a subposet of the upper interval:

\begin{lemma}
\label{lem:fibreonlytwo_subseteq_upper_fibre}
We have $\fibreonlytwo\subseteq \fibrealltwo$.
\end{lemma}
\begin{proof}
	Let $\tau\in\fibreonlytwo$. Then by definition, $\tau\in \FreeS$ and $\rho \subset \tau$. 
	Furthermore, we have $\tau\setminus \rho \subseteq \twosepedges(\tau)$, so by \cref{lem:2_sep_collapsed_reduced_size},
	\begin{equation*}
		p = |\twosepedges(\rho)|< |\twosepedges(\tau)|.
	\end{equation*}
	By \cref{lem:characterisation_upper_fibre}, this shows that $\tau\in \fibrealltwo$. 
\end{proof}

The next observation we make is that when passing from $\tau\in \fibrealltwo$ to $\rho$, we need to collapse at least one edge that lies in a 2-separator:

\begin{lemma}
	\label{lem:collapsed_things_inside_2_bond}
	Let $\tau \in \fibrealltwo$. Then $\twosepedges(\tau)\not\subseteq \rho$.
\end{lemma}
\begin{proof}
	If $\twosepedges(\tau)\subseteq \rho$, then by \cref{lem:collapsing_not_to_C_preserves_2_cut_edges}, we have $|\twosepedges(\tau)|\leq |\twosepedges(\rho)| = p$. By \cref{lem:characterisation_upper_fibre}, this implies that $\tau\not\in \fibrealltwo$.
\end{proof}

The proof that the upper interval is homotopy equivalent to $\fibreonlytwo$ is more involved than its analogue for the lower fibre and in fact uses the results about the lower fibre we obtained in \cref{sec:lower_fibre}.
\begin{lemma}
\label{lem:upper_fibre_simeq_fibreonlytwo}
We have $\fibrealltwo \simeq \fibreonlytwo$.
\end{lemma}
\begin{proof}
	We will prove this result by considering it as a relative version of \cref{prop:lower_fibre} and mimicking the proof of that lemma in \cref{sec:lower_fibre}.

	We start with a relative version of \cref{def:Ctwobonds}:
	Let $X_{r,s}\subseteq \fibrealltwo$ be the subposet of all $\sigma\in \fibrealltwo$ such that one of the following holds:
	\begin{enumerate}
		\item $\sigma\in \fibreonlytwo$,
		\item $|\twosepedges(\sigma)\setminus \rho| > r$ or
		\item $|\twosepedges(\sigma)\setminus \rho| = r$ and $|\sigma\setminus \rho| \leq  r+s$.
	\end{enumerate}

	We have the following:
	\begin{enumerate}
		\item If $r>2n-3$, then $X_{r,s} = \fibreonlytwo$.
		\item If $s=0$, we have $X_{r,s} = X_{r+1,3n-4-r}$.
		\item If $s\geq 3n-3-|\rho|$, we have $X_{1,s} = \fibrealltwo$.
	\end{enumerate}
	The arguments are analogous to the ones in \cref{lem:filtration_of_Ctwobonds}.

	It follows that we have a filtration 
	\begin{multline*}
		\fibreonlytwo = X_{2n-2,\bullet} = X_{2n-3,0} \subseteq X_{2n-3,1}\subseteq \cdots \subseteq X_{1,3n-3-|\rho|} = \fibrealltwo.
	\end{multline*}
	We will show that for all $r,s$, the inclusion $X_{r,s-1} \hookrightarrow X_{r,s}$ is a homotopy equivalence. This will imply that $\fibreonlytwo \hookrightarrow  \fibrealltwo$ is a homotopy equivalence as well.

	Fix $r,s$, let $\iota\colon X_{r,s-1} \hookrightarrow X_{r,s}$ be the inclusion and  $\tau\in X_{r,s}$. We claim that the fibre $\iota_{\subset \tau}$ is contractible.
	If $\tau \in X_{r,s-1}$, this is clear because $\tau$ is the unique maximal element of $\iota_{\subset \tau}$. 
	So assume that $\tau\in X_{r,s}\setminus X_{r,s-1}$. This means that
	\begin{equation*}
		\tau\in \fibrealltwo \setminus \fibreonlytwo,\, |\twosepedges(\tau)\setminus \rho| = r \text{ and } |\tau\setminus \rho| =  r+s.
	\end{equation*}
	Combining \cref{lem:characterisation_upper_fibre} with an analogue of \cref{lem:characterisation_lower_fibre}, we get that $\sigma\in\iota_{\subset \tau}$ if and only if
	\begin{enumerate}
		\item $\rho\subset\sigma\subset \tau$,
		\item $|\twosepedges(\sigma)| > p$ and
		\item $\sigma \in \fibreonlytwo$ or  $|\twosepedges(\sigma)\setminus \rho|\geq r$.
	\end{enumerate}

	We claim that for all $\sigma\in \iota_{\subset \tau}\cup \{\rho\}$, we also have $\sigma'\coloneqq \sigma\cup \twosepedges(\tau)\in \iota_{\subset \tau}$.
	It is clear that $\rho \subseteq \sigma'\subseteq \tau$ and it follows from \cref{lem:collapsing_not_to_C_preserves_2_cut_edges} that $\twosepedges(\tau)\subseteq \twosepedges(\sigma')$.
	This firstly implies that $|\twosepedges(\sigma')|\geq |\twosepedges(\tau)|>p$ (for the last inequality, we use \cref{lem:characterisation_upper_fibre}) and secondly that $|\twosepedges(\sigma')\setminus \rho|\geq |\twosepedges(\tau)\setminus \rho| = r$.

	It remains to show that $\rho\neq\sigma'\neq \tau$.
	If $\rho\subset \sigma$, it is clear  that also $\rho \subset\sigma'$. If $\sigma= \rho$, then $\sigma'=\rho$ is equivalent to saying that $\twosepedges(\tau)\subseteq \rho$, which cannot be the case by \cref{lem:collapsed_things_inside_2_bond}.

	Now assume that $\sigma' = \tau$, i.e.~$\tau\setminus \sigma \subseteq \twosepedges(\tau)$. This leads to a contradiction, similarly to the proof of \cref{lem:collapsed_things_outside_2_bond}:
	By \cref{lem:2_sep_collapsed_reduced_size}, we get 
	$\twosepedges(\sigma)\subseteq\twosepedges(\tau)$.
	As $\sigma \neq \tau$, this is a proper inclusion, so we have $|\twosepedges(\sigma)\setminus \rho|<|\twosepedges(\tau)\setminus \rho| = r$.
	As $\sigma \in \iota_{\subset \tau}\cup \ls \rho \rs$, the third condition for being in $\iota_{\subset \tau}$ then implies that $\sigma\in \fibreonlytwo\cup \ls \rho\rs$, so $\sigma\setminus \rho\subseteq \twosepedges(\sigma)\subset \twosepedges(\tau)$.
	But as
	\begin{equation*}
		\tau \setminus \rho = (\tau\setminus \sigma) \cup (\sigma\setminus \rho),
	\end{equation*}
	we then also get $\tau\in \fibreonlytwo$. This is a contradiction and implies that $\sigma'\subset \tau$, so we indeed get $\sigma'\in \iota_{\subset \tau}$.

	It follows that the assignment $\sigma\mapsto \sigma'$ defines a monotone poset map $\iota_{\subset \tau} \to \iota_{\subset \tau}$ (this is the analogue of \cref{prop:lower_fibre_simeq_fibrenontwo}). 
	Its image is contractible because it has a unique minimal element, namely $\rho' \coloneqq \rho\cup\twosepedges(\tau)$ (this is the analogue of \cref{lem:fibrenontwo_contractible}).
	Hence, also $\iota_{\subset \tau}$ is contractible, which is what we wanted to show.
\end{proof}

By \cref{lem:upper_fibre_simeq_fibreonlytwo}, we can prove \cref{prop:upper_fibre} by determining the homotopy type of $\fibreonlytwo$. We will do this in the next section.

\section{Homotopy type of the upper-interval model $\fibreonlytwo$}

\label{sec:fibreonlytwo}
We work towards showing that for $p=0$, $\fibrealltwo\simeq\fibreonlytwo$ is a wedge of $(n-3)$-spheres if $\Gamma(\rho)$ is a $\theta$-graph, and is contractible otherwise.
Before specialising to $p=0$, we describe some structural properties of $\fibreonlytwo$ that do not depend on $p$.

\subsection{The complex $\fibreonlytwo$}
Recall from \cref{sec:link_sphere_system} that the poset $\FreeS_{\supset\rho}$ can be seen as a simplicial complex using the isomorphism to the link $\lk_{\FreeS}(\rho)$ of $\rho$ in the simplicial complex $\FreeS$.

\subsubsection{$\fibreonlytwo$ is a flag complex}

The upper interval $\fibrealltwo$ is contained in $\FreeS_{\supset\rho}$, but it is just a subposet and not a subcomplex, as it is not downwards closed.
An example is depicted in the top row of \cref{fig:not_full_subcomplex}: The sphere system $\tau = \rho \cup \ls S,T,R\rs$ is in $\fibrealltwo$ because
\begin{equation*}
	|\twosepedges(\tau)|=1>\twosepedges(\rho)=0
\end{equation*}
(see \cref{lem:characterisation_upper_fibre}), but the sphere system $\sigma = \rho \cup \ls T,R\rs\subset \tau$ is not because $\twosepedges(\sigma) = \emptyset$.

The next two lemmas show that, in contrast to the upper interval itself, we can see $\fibreonlytwo$ as a subcomplex of $\FreeS_{\supset\rho}$. It is not a full subcomplex: if every vertex of a simplex $\sigma\in\FreeS_{\supset\rho}$ is contained in $\fibreonlytwo$, this does not necessarily mean that $\sigma\in\fibreonlytwo$; an example is depicted in the bottom row of \cref{fig:not_full_subcomplex}. 
However, $\fibreonlytwo$ is a flag complex: if every \emph{edge} of a simplex $\sigma\in\FreeS_{\supset\rho}$ is contained in $\fibreonlytwo$, then $\sigma\in\fibreonlytwo$.

\begin{figure}
	\centering
	\includegraphics{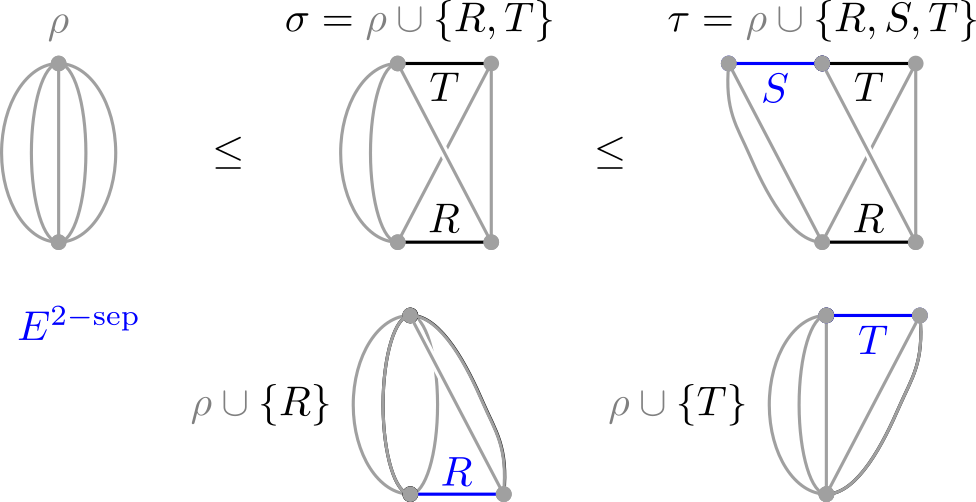}
	\caption{Top row: The dual graphs of two sphere systems $\sigma \subset\tau$ in $\FreeS_{\supset\rho}$, where $\rho$ is the $\theta$-graph in rank $n=5$. We have $\tau\in \fibreonlyone$, but $\sigma \notin \fibreonlytwo$. 
	Bottom row: The vertices $\rho \cup \ls R\rs, \rho\cup \ls T \rs$ of the edge $\sigma$ lie in $\fibreonlytwo$.}
	\label{fig:not_full_subcomplex}
\end{figure}

\begin{lemma}
	\label{lem:fibreonlytwo_downwards_closed}
	If $\tau\in \fibreonlytwo$ and $\rho\subset \sigma \subseteq \tau$, then $\sigma\in \fibreonlytwo$.
\end{lemma}
\begin{proof}
	As $\tau$ is contained in $\fibreonlytwo$, we have
	\begin{equation*}
		(\tau\setminus \sigma) \cup (\sigma\setminus \rho) = \tau \setminus \rho \subseteq \twosepedges(\tau).
	\end{equation*}
	So  $\sigma\setminus \rho\subseteq\twosepedges(\tau)$ and then by \cref{lem:collapsing_not_to_C_preserves_2_cut_edges}, we also have $\sigma\setminus \rho \subseteq \twosepedges(\sigma)$.
\end{proof}

\cref{lem:fibreonlytwo_downwards_closed} implies that $\fibreonlytwo$ is downwards closed in $\FreeS_{\supset\rho}$, so it is a subcomplex. \cref{lem:fibreonlytwo_is_flag} below shows that it is a flag complex.

\begin{lemma}
	\label{lem:fibreonlytwo_is_flag}
	Let $\rho\subset \sigma \in \FreeS$. Then $\sigma\in \fibreonlytwo$ if and only if for all $S,T\in \sigma\setminus \rho$, we have $\rho\cup \ls S,T \rs\in \fibreonlytwo$.
\end{lemma}
\begin{proof}
	The ``only if'' direction is \cref{lem:fibreonlytwo_downwards_closed}. 
	For the opposite direction, assume that for all $S,T\in \sigma\setminus \rho$, we have $\rho\cup \ls S,T \rs\in \fibreonlytwo$. We need to show that $\sigma\in \fibreonlytwo$, i.e.~ that $\sigma\setminus \rho \subseteq \twosepedges(\sigma)$. This is \cref{lem:2_sep_pairwise}.
\end{proof}

From now on, we will identify $\fibreonlytwo$ with this flag subcomplex of $\lk_{\FreeS}(\rho)$. That is, we consider it as a simplicial complex, where the vertices are isotopy classes of spheres $S\not \in \rho$ such that $\rho \cup \ls S \rs \in \FreeS$ and $S\in \twosepedges(\rho \cup \ls S \rs)$.
A collection of such spheres forms a simplex if and only if for all $S,T$ in the collection, we have $\rho\cup \ls S,T \rs\in \FreeS$ and $ \ls S,T \rs\subseteq \twosepedges(\rho\cup \ls S,T \rs)$.

\subsubsection{A join decomposition}

In what follows, we set
\begin{equation*}
	G\coloneqq \Gamma(\rho) \text{ and write } G_S\coloneqq \Gamma(\rho\cup \ls S \rs),\, G_{ST}\coloneqq \Gamma(\rho\cup \ls S, T \rs)
\end{equation*}
if $\rho\cup \ls S\rs, \, \rho\cup \ls S, T \rs\in \FreeS$.

Recall from \cref{sec:link_sphere_system} that the link $\lk_{\FreeS}(\rho)$ decomposes as a join 
\begin{equation*}
	\lk_{\FreeS}(\rho) = \ast_{v\in V(G)} \FreeS_v,
\end{equation*}
where $\FreeS_v$ is the subcomplex of $\lk_{\FreeS}(\rho)$ consisting of all spheres $S$ that are contained in the connected component $M_v$ of $M-\rho$ corresponding to the vertex $v\in V(G)$.
Let
\begin{equation*}
	\twosepposet_v\coloneqq \FreeS_v\cap \fibreonlytwo
\end{equation*}
be the subcomplex of all sphere systems $\sigma_v$ in $M_v$ that are contained in $\fibreonlytwo$.
Given any $\sigma \in \fibreonlytwo$, we can uniquely decompose it as $\sigma = \bigsqcup_{v\in V(G)}\sigma_v$, where  $\sigma_v\in \twosepposet_v$. 
However, it is not true that $\fibreonlytwo$ also decomposes as a join over all the $\twosepposet_v$. The reason is that it is not a full subcomplex of $\FreeS_{\supset\rho}$.
Before we continue, we note that in contrast to that, the subcomplexes $\twosepposet_v$ are full.

\begin{lemma}
	\label{lem:P_v_full_subcomplex}
	$\twosepposet_v$ is a full subcomplex of $\FreeS_v$.
\end{lemma}
\begin{proof}
	By \cref{lem:fibreonlytwo_is_flag}, it suffices to show that if $S,T\in \twosepposet_v$ form an edge in $\FreeS_v$, then they also form an edge in $\twosepposet_v$. This means that $S,T\in \twosepedges(\rho \cup \ls S,T \rs)$.
	By assumption, we have $S \in \twosepedges(\rho \cup \ls S \rs)$. Hence by \cref{lem:equivalent_characterisations_of_E2sep}, there is a vertex $w$ such that $G_S- \ls e_S, w\rs$ is disconnected.
	As $S$ and $T$ both lie in $\FreeS_v$, it is easy to see that $G_{ST}- \ls e_S, w\rs$ is disconnected as well, see \cref{fig:vw_vu_compatible}.
	\begin{figure}
		\centering
		\includegraphics{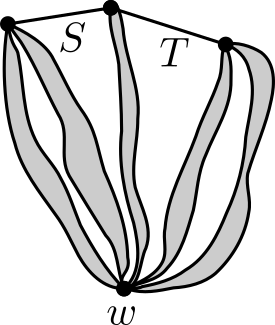}
		\caption{If $S,T$ form an edge in $\FreeS_v$, then they also form an edge in $\twosepposet_v$.}
		\label{fig:vw_vu_compatible}
	\end{figure}
	Hence, we have $S\in \twosepedges(\rho \cup \ls S,T \rs)$.
	Exchanging the roles of $S$ and $T$, we also get that $T\in \twosepedges(\rho \cup \ls S,T \rs)$.
\end{proof}

We next describe subcomplexes of $\twosepposet_v$ that will be useful for understanding when vertices in $\fibreonlytwo$ form a simplex.
For $\ls v,w\rs \subseteq V(G)$, define 
\begin{equation*}
	\twosepposet_{vw} \coloneqq \ls \ls S_0, \ldots, S_k\rs \in \twosepposet_v \mid |G_{S_i}|- \ls e_{S_i},w \rs \text{ is disconnected for all } i \rs
\end{equation*}
to be the full subcomplex of $\twosepposet_v$ on those spheres $S\in \twosepposet_v$ such that $|G_S|- \ls e_S,w \rs$ is disconnected, where $e_S\in E(G_S)$ is the edge corresponding to $S$.

\begin{lemma}
	\label{lem:Pv_union_of_Pvw}
	Every $S\in \twosepposet_v$ is contained in $\twosepposet_{vw}$ for some 2-vertex cut $\ls v,w \rs\subseteq V(G)$.
\end{lemma}
\begin{proof}
	Let $e_S\in \twosepedges(G_S)$ be the corresponding edge in $G_S$. By \cref{lem:equivalent_characterisations_of_E2sep}, there is a vertex $w$ in $G_S$ that is not an endpoint of $e_S$ such that $|G_S|- \ls e_S,w \rs$ has two connected components that each consist of more than one vertex. This implies that also $|G_S|- \ls \overbar{e_S},w \rs$ is disconnected.
	As $w$ is not an endpoint of $e_S$, we can also regard it as a vertex of $G$ and we have a homeomorphism $|G_S|- \ls \overbar{e_S},w \rs \cong |G|- \ls v,w \rs$.
\end{proof}

\subsection{Partition models and compatibility}
\label{sec:partitions}
We want to describe $\twosepposet_{v}$ in terms of the subcomplexes $\twosepposet_{vw}$. In order to do so, we use the combinatorial description of $\FreeS_v$ from \cref{sec:link_sphere_system} that describes $\FreeS_v$ in terms of partitions of $E(v)$, the set of half-edges adjacent to $v$.\footnote{Note that here, we consider only graphs in $\FreeS\setminus \Conebonds$, so there are no loops and we can consider $E(v)$ as a subset of $E(G)$.}

\subsubsection{Structure of a single $\twosepposet_{vw}$}
We now describe the subcomplexes $\twosepposet_{vw}$ in terms of these partitions:
Let $\ls v,w \rs\subseteq V(G)$ be a 2-vertex cut.
A \emph{$vw$-component} of $G$ is a maximal subset of $E(v)$ consisting of edges that all lie in the same connected component of $|G|- \ls v,w \rs$.
The set of all $vw$-components forms a partition of $E(v)$. As $\ls v,w \rs$ is a 2-vertex cut, this partition has at least two elements.

\begin{lemma}
	\label{lem:P_vw_via_partitions}
	Let $S\in \FreeS_{v}$.
	We have $S\in \twosepposet_{vw}\subseteq \twosepposet_{v}$ if and only if each side of the induced partition $P_S$ of $E(v)$ is a non-empty union of $vw$-components.
\end{lemma}
\begin{proof}
	Let $e_S\in E(G_S)$ be the edge corresponding to $S$. 
	By definition, we have $S\in \twosepposet_{vw}$ if and only if $|G_S|- \ls e_S, w \rs$ is disconnected.
	If each side of the induced partition $P_S$ is a non-empty union of $vw$-components, this is certainly the case.

	For the converse, assume that there are $e_1,e_2$ that lie in the same $vw$-component but on distinct sides of $P_S$. Then the component of $G- \ls v,w \rs$ that contains $e_1$ and $e_2$ gives a path between the endpoints of $e_S$ in $G_S- \ls w \rs$. As $G_S\in \FreeS\setminus \Conebonds$, it has no cut vertex, so $G_S- \ls w \rs$ is connected. But then also $G_S- \ls e_S, w \rs$ must be connected, which shows that $S\not\in \twosepposet_{vw}$.
\end{proof}

We summarise the description of $\twosepposet_{vw}$ in terms of partitions as follows: 

\begin{corollary}
	\label{cor:partition_description_P_vw}
	Let $\ls v,w \rs\subseteq V(G)$ be a 2-vertex cut. Write the set of $vw$-components as $X = \ls c_1,\ldots ,c_r, b_1,\ldots, b_s \rs$, where each $c_i\subseteq E(v)$ contains at least two edges and each $b_j \subseteq E(v)$ consists of a single edge.
	Then $\twosepposet_{vw}$ is isomorphic to the simplicial complex that has as vertices the two-element partitions $P$ of $X$ such that no side of $P$ is empty or consists of a single $b_j$, and where a collection of such partitions forms a simplex if and only if they are pairwise compatible.
\end{corollary}
\begin{proof}
	This follows from the description of $\FreeS_v$ in \cref{sec:link_sphere_system}, together with \cref{lem:P_v_full_subcomplex} and \cref{lem:P_vw_via_partitions}.
\end{proof}

\begin{lemma}
	\label{lem:cone_point_P_vw}
	Let $\ls v,w \rs\subseteq V(G)$ be a 2-vertex cut such that $\twosepposet_{vw}\neq \emptyset$ and such that there is at least one $vw$-component with at least two edges.
	Then $\twosepposet_{vw}$ has a cone point.
\end{lemma}
\begin{proof}
	We use the partition description of $\twosepposet_{vw}$ from \cref{cor:partition_description_P_vw}.
	Let $P$ be the partition of $X = \ls c_1,\ldots ,c_r, b_1,\ldots, b_s \rs$ that has $c_1$ on one side and all the other $vw$-components on the other side (by assumption, we have $r\geq 1$).
	The condition that $\twosepposet_{vw}\neq \emptyset$ means that one of the following is true:
	\begin{itemize}
		\item $r\geq 2$,
		\item $s\geq 4$ or
		\item $r\geq 1$ and $s\geq 2$.
	\end{itemize}
	This implies that $P$ is indeed a vertex of $\twosepposet_{vw}$, as no side of $P$ is empty or consists of a single $b_j$.
	Furthermore, $P$ is clearly compatible with all other two-element partitions of $X$ with non-empty sides, so it forms a cone point of $\twosepposet_{vw}$ (cf.~\cite[Section 2]{Conant2005}).
\end{proof}

\cref{lem:cone_point_P_vw} gives us a cone point for a single $\twosepposet_{vw}$. We want to show that this is in fact a cone point for all of $\fibreonlytwo$. To do so, we need to understand better when vertices from different $\twosepposet_{vw}$ form a simplex. This is what we do next.

\subsubsection{Compatibility between different $\twosepposet_{vw}$ at a single $\twosepposet_v$}

Using the partition description from \cref{cor:partition_description_P_vw}, we can show that vertices of $\twosepposet_v$ that lie in distinct $\twosepposet_{vw}$ always form a simplex. The graph-theoretic reason for this is as follows.

\begin{lemma}
	\label{lem:vw_vu_compatible}
	Let $\ls v,w \rs$ and $\ls v,u \rs$ be 2-vertex cuts in $G$ with $w\neq u$. Then there is a $vw$-component $K\subset E(v)$ such that $E(v)\setminus K$ is contained in a $vu$-component.
\end{lemma}
\begin{proof}
	As $G$ is connected, every component of $|G|- \ls v,w \rs$ is adjacent to at least one of $v$ and $w$. And as $G$ is 2-vertex connected, so has no cut vertex, every component needs to be adjacent to both.

	There is a unique connected component of $|G|- \ls v,w \rs$ containing $u$. Let $K$ be the $vw$-component corresponding to this.
	Let $e_1, e_2\in E(v)\setminus K$. As every component of $|G|- \ls v,w \rs$ is adjacent to both $v$ and $w$, there are $vw$-paths $p_1$ and $p_2$ in $|G|- \ls u \rs$ that start with $e_1$ and $e_2$. Then the concatenation of $p_1$ and the inverse of $p_2$ gives a path in $|G|-\ls v,u \rs$ that connects $e_1$ to $e_2$ (see \cref{fig:vw_vu_comps}). Hence, they lie in the same $vu$-component.
	\begin{figure}
		\centering
		\includegraphics{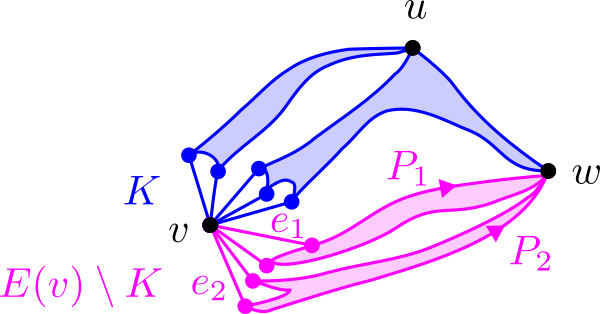}
		\caption{The $vw$-component $K$ containing $u$ has the property that $E(v)\setminus K$ is contained in a single $vu$-component.}
		\label{fig:vw_vu_comps}
	\end{figure}
\end{proof}

\begin{corollary}
	\label{lem:helper1}
	If $S\in \twosepposet_{vw}$ and $T\in \twosepposet_{vu}$, $u\neq w$, then $\ls S, T \rs \in \fibreonlytwo$.
\end{corollary}
\begin{proof}
	By \cref{lem:P_v_full_subcomplex}, it is enough to show that $S$ and $T$ form an edge in $\FreeS_v$, i.e.~that $\rho\cup \ls S,T \rs \in \FreeS$.
	This is the case if and only if the partitions $P_S$ and $P_T$ of $E(v)$ induced by $S$ and $T$ are compatible with one another.
	By \cref{lem:P_vw_via_partitions}, each side of $P_S$ is a union of $vw$-components and each side of $P_T$ is a union of $vu$-components.
	Let $K$ be as in \cref{lem:vw_vu_compatible}. Then the side of $P_S$ not containing $K$ is contained in the single $vu$-component $E(v)\setminus K$ and hence in one of the sides of $P_T$. This shows that $P_S$ and $P_T$ are compatible.
\end{proof}

\subsubsection{Compatibility between $\twosepposet_v$ and $\twosepposet_w$}

Next, we describe when vertices from $\twosepposet_v$ form edges with vertices from $\twosepposet_w$ for some $w\neq v$.

\begin{lemma}
	\label{lem:helper2_1}
	Let $S\in \twosepposet_{v}$ and $T\in \twosepposet_{w}\setminus \twosepposet_{wv}$, where $w\neq v$.
	Then $\ls S, T \rs \in \fibreonlytwo$.
\end{lemma}
Note that $T\in \twosepposet_{w}\setminus \twosepposet_{wv}$ is automatically the case if $\twosepposet_{wv} =\emptyset$, which happens e.g.~if $\ls v,w \rs$ is not a 2-vertex cut.
\begin{proof}
	As $T\in \twosepposet_{w}\setminus \twosepposet_{wv}$, there is $u\neq v$ such that $|G_T|- \ls e_T, u\rs$ is disconnected. But then as $u\neq v$, we have that $u$ is also a vertex in  $G_{ST}$ and $|G_{ST}|- \ls e_T, u\rs$ is disconnected. Hence, $T\in \twosepedges(\rho\cup\ls S, T\rs)$. By assumption, $S\in \twosepedges(\rho\cup\ls S\rs)$. So using \cref{lem:adding_in_2_sep}, we also get that $S\in \twosepedges(\rho\cup\ls S, T\rs)$. It follows that $\ls S, T \rs \in \fibreonlytwo$.
\end{proof}

To describe the compatibility between elements from $\twosepposet_{vw}$ and $\twosepposet_{wv}$, we restrict to the case $p=0$ (i.e.~$\twosepedges(G) = \emptyset$) to avoid too many case distinctions.
In this case, the form of $vw$-components is simpler to describe:

\begin{lemma}
	\label{lem:vw_comps_for_p0}
	Let $\ls v,w \rs\subseteq V(G)$ be a 2-vertex cut and assume that $\twosepedges(G) = \emptyset$.
	Then every component of $|G|\setminus \ls v, w\rs$ either consists of (the interior of) a single edge or it is adjacent to both $v$ and $w$ with at least two edges.
\end{lemma}
\begin{proof}
	Assume that $b$ is a component of $|G|\setminus \ls v, w\rs$ that is adjacent to $v$ with a single edge $e$. If $b$ consisted of more than just the interior of $e$, then $b\setminus \ls e \rs$ would be a non-empty component of $|G|-\ls w,e \rs$ and $w$ would not be an endpoint of $e$. Hence by \cref{lem:equivalent_characterisations_of_E2sep}, we would get $e\in \twosepedges(G)$, which is a contradiction.
\end{proof}

\begin{lemma}
	\label{lem:iso_Pvw_Pwv}
	Let $\ls v,w \rs\subseteq V(G)$ be a 2-vertex cut and assume that $\twosepedges(G) = \emptyset$.
	Then there is a (canonical) isomorphism $\phi_{vw}: \twosepposet_{vw} \to \twosepposet_{wv}$.
\end{lemma}
\begin{proof}
	This follows from the description of $\twosepposet_{vw}$ in terms of partitions in \cref{cor:partition_description_P_vw} and the description of the $vw$-components from \cref{lem:vw_comps_for_p0}.
	Both the $vw$-components and the $wv$-components are in one-to-one correspondence with the connected components of $|G|\setminus \ls v,w \rs$, so they are in one-to-one correspondence with one another.
	\cref{lem:vw_comps_for_p0} implies that this correspondence sends $vw$-components that contain at least two edges to $wv$-components that contain at least two edges, and it sends $vw$-components that consist of a single edge to $wv$-components that consist of a single edge.
	Define $\phi_{vw}$ to be the induced map on partitions of such components. Then clearly, $\phi_{vw}$ is a bijection on vertices.
	But it is also clear that $\phi_{vw}$ preserves compatibility of partitions, so it is a simplicial map and hence an isomorphism of simplicial complexes.
\end{proof}

It is not hard to see that if $S\in \twosepposet_{vw}$, then $\rho \cup \ls S, \phi_{vw}(S) \rs\in \FreeS\setminus \Conebonds$ and the edges corresponding to $S$ and $\phi_{vw}(S)$ form a 2-bond in $\Gamma(\rho\cup \ls S, \phi_{vw}(S) \rs)$.
In fact, one can check that all 2-bonds that one can introduce in $G$ correspond to pairs of spheres that are of this form.
Conceptually, this is why the following holds true:

\begin{lemma}
	\label{lem:helper2_2}
	Let $S\in \twosepposet_{vw}$ and $T\in \twosepposet_{wv}$. The following are equivalent:
	\begin{enumerate}
		\item $\ls S, T \rs \in \fibreonlytwo$
		\item $\ls\phi_{vw}(S), T\rs \in \FreeS_w$
		\item $\ls \phi_{vw}(S), T \rs \in \twosepposet_{wv}$.
	\end{enumerate}
\end{lemma}
\begin{proof}
	It is clear that the third item implies the second. 
	That the second item also implies the third follows from \cref{lem:P_v_full_subcomplex}.

	To see that the first item implies the second one, assume that the second item does not hold, i.e.~$\ls\phi_{vw}(S), T\rs \not\in \FreeS_w$.
	Let $X$ be the set of connected components of $|G|\setminus \ls v,w \rs$. Since $S\in \twosepposet_{vw}$ and $T\in \twosepposet_{wv}$, these correspond to two-element partitions $P_S = \ls A_S, B_S \rs$ and $P_T = \ls A_T, B_T \rs$ of $X$. As $\ls\phi_{vw}(S), T\rs \not\in \FreeS_w$, these partitions are not compatible, so all four intersections
	\begin{equation*}
		A_S\cap A_T,\, A_S\cap B_T,\, B_S\cap A_T,\, B_S\cap B_T
	\end{equation*}
	are non-empty.
	This implies that in $G_{ST} = \Gamma(\rho\cup \ls S, T \rs)$, there are two disjoint paths between the endpoints of the edge $e_S$ corresponding to $S$, see \cref{fig:ST_not_in_two_sep}. By \cref{lem:equivalent_characterisations_of_E2sep}, this shows that $S\not\in \twosepedges(\rho\cup \ls S, T \rs)$. (And similarly for $T$.) 
	Hence, the first item does not hold.
	\begin{figure}
		\centering
		\includegraphics{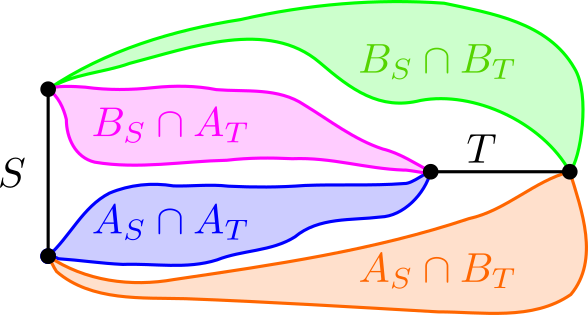}
		\caption{The partitions $P_S = \ls A_S, B_S \rs$ and $P_T = \ls A_T, B_T \rs$ are compatible with each other and hence give rise to disjoint paths between the endpoints of $e_S$ and $e_T$.}
		\label{fig:ST_not_in_two_sep}
	\end{figure}
	
	To see that the second item implies the first one, assume that $\ls\phi_{vw}(S), T\rs \in \FreeS_w$. As noted above, we have $\rho \cup \ls \phi_{vw}(S), S \rs\in \FreeS\setminus \Conebonds$ and  $\phi_{vw}(S)$ and $S$ form a 2-bond in it. This implies that they also form a 2-bond in $\rho \cup \ls S, \phi_{vw}(S), T \rs$, so 
	\begin{equation*}
		\ls S, \phi_{vw}(S) \rs\in \twosepedges(\rho \cup \ls S, \phi_{vw}(S), T \rs).
	\end{equation*}
	As $T\in \twosepedges(\rho \cup \ls T \rs)$, \cref{lem:2_sep_collapsed_reduced_size} implies that also $T\in \twosepedges(\rho \cup \ls S, \phi_{vw}(S), T \rs)$. Hence, $\ls S,\phi_{vw}(S),  T \rs\in \fibreonlytwo$. By \cref{lem:fibreonlytwo_downwards_closed}, we then also have $\ls S, T \rs\in \fibreonlytwo$.	
\end{proof}

\subsection{Homotopy type of \texorpdfstring{$\fibreonlytwo$}{C'}}

We now assemble the above results to describe the homotopy type of $\fibreonlytwo$ for $p=0$.

\begin{lemma}
	\label{lem:2_separators_in_Ctwobonds}
	If $\twosepedges(\rho) = \emptyset$, then $G = \Gamma(\rho)$ contains a 2-vertex cut $\ls v,w \rs$ such that $\twosepposet_{vw}\neq \emptyset$.
\end{lemma}
\begin{proof}
	As $\rho\in \Ctwobonds$, there is $\sigma\supseteq \rho$ such that $\sigma$ contains $S, T$ that define a 2-bond in $\Gamma(\sigma)$. In particular, $S,T\in \twosepedges(\sigma)$ and hence by \cref{lem:collapsing_not_to_C_preserves_2_cut_edges}, we also have $S\in \twosepedges(\rho \cup \ls S \rs)$. As $\twosepedges(\rho) = \emptyset$, \cref{lem:collapsing_not_to_C_preserves_2_cut_edges} implies that $S\not\in \rho$.
	Let $v$ be the vertex in $G $ that $S$ collapses to. Then we get $S\in \twosepposet_v$. Hence by \cref{lem:Pv_union_of_Pvw}, we get $S\in \twosepposet_{vw}$ for some 2-vertex cut $\ls v,w \rs$.
\end{proof}

\begin{lemma}
	\label{lem:contractible_upper_fibre}
	Assume that $\twosepedges(\rho) = \emptyset$ and that $G = \Gamma(\rho)$ is not a $\theta$-graph. Then $\fibreonlytwo$ has a cone point, so is contractible.
\end{lemma}
\begin{proof}
	By \cref{lem:2_separators_in_Ctwobonds}, there is a 2-vertex cut $\ls v,w \rs$ in $G$ such that $\twosepposet_{vw}\neq \emptyset$.
	By \cref{lem:vw_comps_for_p0}, every component of $|G|\setminus \ls v,w \rs$ is either a single edge or is adjacent to both $v$ and $w$ with at least two edges. 
	If every component consisted of a single edge, then $G$ would be a $\theta$-graph, which we excluded by assumption. Hence, there is at least one component that is adjacent to both $v$ and $w$ with at least two edges. By \cref{lem:cone_point_P_vw}, this implies that $\twosepposet_{vw}$ has a cone point.
	We want to show that this cone point is also a cone point for $\fibreonlytwo$.
	By \cref{lem:fibreonlytwo_is_flag}, it is enough to show that for every $T\in \fibreonlytwo$, we have $\ls S,T \rs \in \fibreonlytwo$.

	First assume that $T\in \twosepposet_v$. Then if $T\in \twosepposet_{vw}$, we have $\ls S,T \rs \in \fibreonlytwo$ as $S$ is a cone point of $\twosepposet_{vw}$. If $T\in \twosepposet_v\setminus\twosepposet_{vw}$, then by \cref{lem:Pv_union_of_Pvw}, we have $T\in \twosepposet_{vu}$ for some 2-vertex cut $\ls v,u \rs$ with $u\neq w$. Then by \cref{lem:helper1}, we have $\ls S,T \rs \in \fibreonlytwo$.

	Now assume that $T\not\in \twosepposet_v$. If $T\not \in \twosepposet_{wv}$, then by \cref{lem:helper2_1}, we have $\ls S,T \rs \in \fibreonlytwo$. If $T\in \twosepposet_{wv}$, then by \cref{lem:helper2_2}, we have $\ls S,T \rs \in \fibreonlytwo$ if and only if $\phi_{vw}(S)\cup T\in \twosepposet_{wv}$. But as $S$ is a cone point of $\twosepposet_{vw}$ and $\phi_{vw}$ is an isomorphism, we know that $\phi_{vw}(S)$ is a cone point of $\twosepposet_{wv}$. Hence, $\phi_{vw}(S)\cup T\in \twosepposet_{wv}$.
\end{proof}

It remains to consider the case where $G$ is a $\theta$-graph. Here, we do not get a cone point, but we can still describe the homotopy type of $\fibreonlytwo$ explicitly. 

\begin{lemma}
	\label{lem:fibre_theta_graph}
	If $G = \Gamma(\rho)$ is a $\theta$-graph with vertex set $V(G) = \ls v,w\rs$, then $\fibreonlytwo\simeq \twosepposet_{vw}$.
	This complex is homotopy equivalent to a wedge of $(n-1)!$ spheres of dimension $n-3$.
\end{lemma}
\begin{proof}
	In this case, the vertex set of $\fibreonlytwo$ is given as the disjoint union of $\twosepposet_{v} = \twosepposet_{vw}$ and $\twosepposet_{w} = \twosepposet_{wv}$. By \cref{lem:iso_Pvw_Pwv}, there is an isomorphism $\phi_{vw}: \twosepposet_{vw} \to \twosepposet_{wv}$.
	The complex $\fibreonlytwo$ contains both $\twosepposet_{v}$ and $\twosepposet_{w}$ as (full) subcomplexes.
	Hence, we can write every $\sigma\in \fibreonlytwo$ uniquely as $\sigma = \sigma_v \cup \sigma_w$ with $\sigma_v\in \twosepposet_{v}$ and $\sigma_w\in \twosepposet_{w}$.
	Define
	\begin{equation*}
		\fibreonlytwo\to \fibreonlytwo
	\end{equation*}
	by sending $\sigma_v\cup \sigma_w$ to $\left(\sigma_v\cup\phi_{wv}(\sigma_w)\right)\cup \left(\sigma_w\cup \phi_{vw}(\sigma_v)\right)$.
	That this map is well-defined follows from \cref{lem:helper2_2} and \cref{lem:fibreonlytwo_is_flag}.
	This defines a monotone map on the poset of simplices of $\fibreonlytwo$, hence a homotopy equivalence to its image. But this image consists of all pairs of the form $\sigma_v \cup \phi_{vw}(\sigma_v)$ and hence is isomorphic to $\twosepposet_{vw}$.
	
	The homotopy type of $\twosepposet_{vw}$ can be computed using the partition description from \cref{cor:partition_description_P_vw}. Vogtmann showed that this complex is homotopy equivalent to a wedge of $(n-1)!$ spheres of dimension $n-3$ \cite[Proof of Theorem 2.4]{Vog:Localstructuresome}, see \cite[Section 3.2]{Billera2001}. In fact, this poset is Cohen--Macaulay \cite{Robinson1996}, even shellable \cite{Trappmann1998}.
\end{proof}

The previous two lemmas now give a complete description of the homotopy types of the upper interval and prove \cref{prop:upper_fibre}:
\begin{proof}[Proof of \cref{prop:upper_fibre}] 
	Using the homotopy equivalence $(\Ctwobonds_{p,q})_{\supset \rho} = \fibrealltwo \simeq \fibreonlytwo$ from \cref{lem:upper_fibre_simeq_fibreonlytwo}, \cref{prop:upper_fibre} immediately follows from \cref{lem:contractible_upper_fibre} and \cref{lem:fibre_theta_graph}.
\end{proof}

\section{Finishing the proof for the 2-bond thickening}
\label{sec:finishing_proof}

Assembling the results about the lower fibres and upper intervals, we can now finish the proof of \cref{thm:C'_he_C}, which states that the inclusion $\Conebonds\hookrightarrow \Ctwobonds$ is $(2n-3)$-connected.

\begin{proposition}
	\label{prop:inclusion_fibres_highly_connected}
	Let $\sigma\in \Ctwobonds_{p,q}\setminus \Ctwobonds_{p,q-1}$. Then the join 
	\begin{equation}
		\label{eq:fibre_join_two_bonds}
			f^{-1}((\Ctwobonds_{p,q})_{\subseteq \sigma})\ast (\Ctwobonds_{p,q})_{\supset \sigma}
	\end{equation}
	is contractible except if $\Gamma(\sigma)$ is a $\theta$-graph, which is the case if and only if $p=0$ and $q = n+1$. 
	In that case, the join is homotopy equivalent to a wedge of $(n-1)!$ spheres of dimension $2n-3$.
\end{proposition}
\begin{proof}
	If $\Gamma(\sigma)$ is not a $\theta$-graph, then the proposition follows immediately from \cref{prop:lower_fibre} and \cref{prop:upper_fibre} because the join of a contractible space with any other space is contractible. 
	
	A $\theta$-graph of rank $n\geq 2$ has two vertices and $n+1$ edges between them, no pair of which forms a 2-bond. So it is clear that these can only arise if $p=0$ and $q=n+1$. Every other graph of rank $n$ with $n+1$ edges has a cut vertex, so lies in $\Conebonds$ and hence not in $\Ctwobonds_{0,n+1}\setminus \Ctwobonds_{0,n}$.
	Lastly, if $\Gamma(\sigma)$ is a $\theta$-graph, then $f^{-1}((\Ctwobonds_{0,n+1})_{\subseteq \sigma})$ is the poset of all proper faces of $\sigma$, because each such face lies in $\Conebonds \subseteq \Ctwobonds_{0,n+1}$. Hence, the order complex of this poset is the barycentric subdivision of the boundary of the $n$-simplex $\sigma$, so homotopy equivalent to an $(n-1)$-sphere. By \cref{lem:fibre_theta_graph}, the upper interval $(\Ctwobonds_{p,q})_{\supset \sigma}$ is homotopy equivalent to a wedge of $(n-1)!$ spheres of dimension $n-3$. Hence, the join \eqref{eq:fibre_join_two_bonds} is homotopy equivalent to a wedge of $(n-1)!$ spheres of dimension $(n-3)+(n-1)+1 = 2n-3$.
\end{proof}

Using the above proposition, we can now finish the proofs of \cref{thm:C'_he_C} and \cref{thm:intro}.

\begin{proof}[Proof of \cref{thm:C'_he_C}]
	We apply Quillen's fibre lemma as stated in \cref{thm:Quillen_fibre}, using \cref{prop:inclusion_fibres_highly_connected}. Consider the inclusions of posets
	\begin{equation*}
		\Conebonds \hookrightarrow \Ctwobonds_{0,n} \hookrightarrow \Ctwobonds_{0,n+1} \hookrightarrow \Ctwobonds.
	\end{equation*}
	The first and third inclusions are homotopy equivalences because the fibre joins are contractible and the second is $(2n-3)$-connected because its fibre joins in \cref{prop:inclusion_fibres_highly_connected} are $(2n-4)$-connected.
\end{proof}

\begin{proof}[Proof of \cref{thm:intro}]
	\cref{thm:intro} follows immediately from \cref{thm:hom_eq_partialS_C}, which states that $\partial \FreeS\hookrightarrow \Conebonds$ is a homotopy equivalence, and \cref{thm:C'_he_C}, which states that $\Conebonds\hookrightarrow \Ctwobonds$ is $(2n-3)$-connected.
\end{proof}

\section{Equivariance and graph complexes}

We conclude with comments on the equivariance of the homotopy equivalences constructed in this article and about the relation to the commutative graph complex. The discussion is brief and only sketches the arguments.
Throughout the section, let $G = \Out(F_n)$ and for $\sigma\in \FreeS$, let $G_{\sigma}$ denote the $G$-stabiliser of $\sigma$.

\subsection{Equivariance}
\label{sec:equivariance}

In \cite[Proposition A.1]{Piterman2025a}, Piterman--Smith give an equivariant version of Quillen's fibre lemma. Using it one obtains the following.

\begin{theorem}
	\label{thm:equivariance}
	For all $(p,q)$, the inclusion $\Conebonds_{p,q-1}\hookrightarrow \Conebonds_{p,q}$ is a $G$-homotopy equivalence.

	For $(p,q)\neq (0,n+1)$, the inclusion $\Ctwobonds_{p,q-1}\hookrightarrow \Ctwobonds_{p,q}$ is a $G$-homotopy equivalence.
\end{theorem}
\begin{proof}
	Let $X$ denote either $\Conebonds$ or $\Ctwobonds$.
	By \cite[Proposition A.1.2]{Piterman2025a}, it suffices to show that for all $\sigma\in X_{p,q}\setminus X_{p,q-1}$, the join 
	\begin{equation*}
			f^{-1}((X_{p,q})_{\subseteq \sigma})\ast (X_{p,q})_{\supset \sigma}
	\end{equation*}
	is $G_\sigma$-contractible.

	We have already shown that in each case the joins are contractible; it therefore remains to verify that all homotopy equivalences employed are equivariant.
	To this end we use the following criteria, which hold for any group $H$:
	\begin{enumerate}
		\item \label{it:equivariant_monotone}Every monotone $H$-equivariant poset map from an $H$-poset to itself defines an $H$-homotopy equivalence to its image.
		\item \label{it:equivariant_unique_min_max}If an $H$-poset has a unique minimal or maximal element, then it is $H$-contractible.
		\item \label{it:equivariant_simplicial}If an $H$-simplicial complex has a cone point, then it is $H$-contractible.
		\item \label{it:equivariant_join}The join of $H$-posets or $H$-simplicial complexes is $H$-contractible if one of the join factors is $H$-contractible.
	\end{enumerate}
	\cref{it:equivariant_monotone} was shown by Th{\'e}venaz--Webb \cite[Corollary 1.2]{Thevenaz1991}, the other ones are easy to derive from it.	

	One can now go through the homotopy equivalences in the previous sections and use these four items to show that they are all equivariant. We will not discuss this in detail, but included an overview of the relevant results in \cref{tab:selected_lemmas_overview}.
	A fact used repeatedly is that for every $g\in G$ and $\sigma\in \FreeS$, we have $g\onesepedges(\sigma) = \onesepedges(g \sigma)$ and  $g\twosepedges(\sigma) = \twosepedges(g \sigma)$.
	\begin{table}
		\centering
		\begin{tabular}{l|l}
			Result & Criteria used to show equivariance \\
			\hline
			\cref{lem:lower_fibre_one_bonds} & \cref{it:equivariant_monotone}, \cref{it:equivariant_unique_min_max} \\
			\hline
			\cref{lem:upper_fibre_one_bonds} & \cref{it:equivariant_monotone}, \cref{it:equivariant_simplicial}, \cref{it:equivariant_join} \\
			\hline
			\cref{prop:lower_fibre_simeq_fibrenontwo} & \cref{it:equivariant_monotone} \\
			\hline
			\cref{lem:fibrenontwo_contractible} & \cref{it:equivariant_unique_min_max} \\
			\hline
			\cref{lem:upper_fibre_simeq_fibreonlytwo} & \cref{it:equivariant_monotone}, \cref{it:equivariant_unique_min_max} \\
			\hline
			\cref{lem:contractible_upper_fibre} & \cref{it:equivariant_simplicial} \\
			\hline
			\cref{prop:inclusion_fibres_highly_connected} & \cref{it:equivariant_join}
		\end{tabular}
		\caption{Homotopy equivalences and the criteria used to establish equivariance.}
		\label{tab:selected_lemmas_overview}
	\end{table}
\end{proof}

\cref{thm:equivariance} implies that the first and the last inclusion in
\begin{equation*}
	\partial \FreeS \hookrightarrow \Ctwobonds_{0,n} \hookrightarrow \Ctwobonds_{0,n+1} \hookrightarrow \Ctwobonds
\end{equation*}
are $G$-homotopy equivalences.
The middle inclusion, however, is not.

\subsection{Relation to the commutative graph complex}
\label{sec:graph_complexes}

In the introduction, \cref{sec:graph_complexes_intro}, it is claimed that conceptually, \cref{thm:intro} can be seen as a universal cover version of results on the commutative graph complex due to Conant--Gerlits--Vogtmann \cite{Conant2005}, Willwacher--{\v{Z}}ivkovi{\'c} \cite{Willwacher2015} and Willwacher \cite{Willwacher2025a}. In the present section, we add more details on this relation and explain what additional input would be needed to actually deduce these results from \cref{thm:intro}.

For $n\geq 3$, the combination of these graph complex results (cf.~\cref{sec:graph_complexes_intro}) is equivalent to the existence of an isomorphism
\begin{equation}
	\label{eq:equivariant_homology_iso}
	H_\bullet^{G}(\FreeS, \partial \FreeS; \coeffs)\cong H_\bullet^{G}(\FreeS, \Ctwobonds; \coeffs),
\end{equation}
where $\coeffs = \mbQ$ with the trivial action for the even version of the commutative graph complex, and $\coeffs = \mbQ^{\det}$ for the odd version. Here, $\mbQ^{\det}$ denotes the determinant representation that lets $\phi\in \Out(F_n)$ act on $\mbQ$ via the concatenation of maps
\[
\begin{array}{ccccc}
\Out(F_n) &\to& \GL{n}{\mbZ} &\to& \ls \pm 1\rs\\
\phi &\mapsto& \bar{\phi} &\mapsto& \det(\bar{\phi}).
\end{array}
\]
The equivariant homology in \cref{eq:equivariant_homology_iso} is group homology with coefficients in a chain complex (see \cite[Chapter VII.7]{Bro:Cohomologygroups}), so
\begin{equation*}
	H_\bullet^{G}(X, Y; \coeffs)\coloneqq H_\bullet(G; C_\bullet(X, Y; \coeffs)).
\end{equation*}

From the short exact sequence 
\begin{equation*}
	0 \to C_\bullet(\Ctwobonds, \partial \FreeS; \coeffs) \to C_\bullet(\FreeS, \partial \FreeS; \coeffs) \to C_\bullet(\FreeS, \Ctwobonds; \coeffs) \to 0,
\end{equation*}
one gets a long exact sequence 
\begin{multline*}
	\cdots \to H_k^{G}(\Ctwobonds, \partial \FreeS; \coeffs)
	\to H_k^{G}(\FreeS, \partial \FreeS; \coeffs) \\
	\to H_k^{G}(\FreeS, \Ctwobonds; \coeffs)
	\to H_{k-1}^{G}(\Ctwobonds, \partial \FreeS; \coeffs) \to \cdots
\end{multline*}
so the difference between the two sides of \cref{eq:equivariant_homology_iso} is measured by the $G$-equivariant homology of the pair $(\Ctwobonds, \partial \FreeS)$ with coefficients in $\coeffs$.
If one can show that this homology vanishes, then one gets the isomorphism in \cref{eq:equivariant_homology_iso}.
To relate this to \cref{thm:intro}, one can look at the hyperhomology spectral sequence for $H^{G}_{\bullet}(\Ctwobonds, \partial \FreeS; \coeffs)$, see \cite[Chapter VII.5]{Bro:Cohomologygroups}. It has an $E^2$-page of the form
\begin{equation*}
	E^2_{p,q} = H_p\left(G; H_q(\Ctwobonds, \partial \FreeS; \coeffs)\right) \Longrightarrow H^{G}_{p+q}(\Ctwobonds, \partial \FreeS; \coeffs).
\end{equation*}
So if one can show that the (non-equivariant) relative homology $H_q(\Ctwobonds, \partial \FreeS; \coeffs)$ vanishes for all $q$, then one gets the desired vanishing of the $G$-equivariant homology of $(\Ctwobonds, \partial \FreeS)$. Now if the inclusion $\partial \FreeS\hookrightarrow \Ctwobonds$ were a homotopy equivalence, then it would immediately follow that $H_q(\Ctwobonds, \partial \FreeS; \coeffs)$ vanishes for all $q$.\footnote{Note that in order to get this vanishing, one does not need $G$-equivariance of the homotopy equivalence.} 
However, \cref{thm:intro} only states that it is $(2n-3)$-connected, so one a priori only gets vanishing for $q\leq 2n-3$ and additional work is needed to deduce \cref{eq:equivariant_homology_iso}.

We sketch how the information about the non-trivial fibres of $\partial \FreeS\hookrightarrow \Ctwobonds$ we obtained in \cref{prop:inclusion_fibres_highly_connected} can be helpful here:
We know that $\partial \FreeS$ is homotopy equivalent to a complex of dimension $2n-3$ \cite[Theorem B]{BG:Homotopytypecomplex}. Assume that the same was true for $\Ctwobonds$ -- note that the results in this article do not prove this, but it seems like a natural hypothesis and simplifies the following discussion. Then the relative homology $H_q(\Ctwobonds, \partial \FreeS; \coeffs)$ vanishes for all $q\neq 2n-2$. 
This implies that the spectral sequence above has only one non-zero row, so it collapses and one gets
\begin{equation*}
	H^{G}_{p+2n-2}(\Ctwobonds, \partial \FreeS; \coeffs) \cong H_p\left(G; M\right),
\end{equation*}
where $M\coloneqq H_{2n-2}(\Ctwobonds, \partial \FreeS; \coeffs)$.
By our dimension assumption on $\Ctwobonds$, the module $M$ sits in a short exact sequence
\begin{equation*}
	0 \to M \to H_{2n-3}(\partial \FreeS; \coeffs) \to H_{2n-3}(\Ctwobonds; \coeffs) \to 0.
\end{equation*}
So $M$ can be described by understanding the homology of $\partial \FreeS$ in the  top-degree $2n-3$ that is not present in the homology of $\Ctwobonds$.
This contribution comes from the non-trivial fibres that occur when adding $\theta$-graphs. Indeed, by \cref{thm:equivariance}, the inclusions $\partial \FreeS \hookrightarrow \Ctwobonds_{0,n}$ and $\Ctwobonds_{0,n+1}\hookrightarrow \Ctwobonds$ are $G$-homotopy equivalences, being compositions of the $G$-homotopy equivalences provided by that theorem. Hence, as $G$-modules, $M \cong H_{2n-2}(\Ctwobonds_{0,n+1}, \Ctwobonds_{0,n}; \coeffs)$, i.e.~$M$ is concentrated exactly at the filtration step where the $\theta$-graphs are added.
Describing it as a $G$-module requires one to study the action of $G_\sigma$ on the homology of the join
\begin{equation*}
		f^{-1}((\Ctwobonds_{0,n+1})_{\subseteq \sigma})\ast (\Ctwobonds_{0,n+1})_{\supset \sigma},
\end{equation*}
where $\Gamma(\sigma)$ is a $\theta$-graph.
For the homology of the lower fibre 
\begin{equation*}
	\tilde H_k\left(f^{-1}((\Ctwobonds_{0,n+1})_{\subseteq \sigma})\right) \cong
	\begin{cases}
		\mbZ & \text{if } k = n-1, \\
		0 & \text{otherwise,}
	\end{cases}
\end{equation*}
this action is easy to describe because  $f^{-1}((\Ctwobonds_{0,n+1})_{\subseteq \sigma})$ is given by the boundary of an $n$-simplex with vertices the elements of $\sigma$. For the homology of the upper interval, 
\begin{equation*}
	\tilde H_k\left((\Ctwobonds_{0,n+1})_{\supset \sigma}\right) \cong
	\begin{cases}
		\mbZ^{(n-1)!} & \text{if } k = n-3, \\
		0 & \text{otherwise,}
	\end{cases}
\end{equation*}
this is more involved, but should also be doable because this representation has been described explicitly by Robinson--Whitehouse \cite{Robinson1996}.

\printbibliography

\medskip
\noindent\small
\textsc{Benjamin Br\"uck}\\
\textsc{Institut für mathematische Logik und Grundlagenforschung}\\
Universität Münster\\
Einsteinstrasse 62\\
48149 Münster, Germany\\
\href{mailto:benjamin.brueck@uni-muenster.de}{benjamin.brueck@uni-muenster.de}

\end{document}